\documentclass[12pt,leqno,draft]{article}
\usepackage{amsmath, amsthm, amsfonts, amssymb, color, enumerate}
\usepackage{mathrsfs}
\newtheorem{thm}{Theorem}[section]
\newtheorem{cor}[thm]{Corollary}
\newtheorem{lem}[thm]{Lemma}
\newtheorem{prp}[thm]{Proposition}

\theoremstyle{definition}

\newcommand{\scr}[1]{\mathscr #1}
\definecolor{wco}{rgb}{0.5,0.2,0.3}

\numberwithin{equation}{section} \theoremstyle{remark}

\newcommand{\ua}{\uparrow}

\title{
{\bf  Degenerate SDEs with Singular Drift and Applications to Heisenberg Groups}
\footnote{Supported in
 part by  NNSFC (11771326, 11431014).}}
\author{
{\bf Xing Huang$^{a)}$ and Feng-Yu Wang$^{a),b)}$   }\\
\footnotesize{$^{a)}$Center of Applied Mathematics, Tianjin University, Tianjin 300072, China}\\
 \footnotesize{$^{b)}$Department of Mathematics,
Swansea University, Singleton Park, SA2 8PP, United Kingdom}}

\begin{document}
\numberwithin{equation}{section}
\def\theequation{\arabic{section}.\arabic{equation}}

\newcommand{\be}{\begin{eqnarray}}
\newcommand{\ee}{\end{eqnarray}}
\newcommand{\ce}{\begin{eqnarray*}}
\newcommand{\de}{\end{eqnarray*}}
\newtheorem{theorem}{Theorem}[section]
\newtheorem{lemma}[theorem]{Lemma}
\newtheorem{remark}[theorem]{Remark}
\newtheorem{definition}[theorem]{Definition}
\newtheorem{proposition}[theorem]{Proposition}
\newtheorem{Examples}[theorem]{Example}
\newtheorem{corollary}[theorem]{Corollary}

\def\Re{{\mathrm{Re}}}
\def\DD{\Delta}
\def\Im{{\mathrm{Im}}}
\def\var{{\mathrm{var}}}
\def\HS{{\mathrm{\tiny HS}}}
\def\eps{\varepsilon}
\def\t{\tau}
\def\e{\mathrm{e}}
\def\vr{\varrho}
\def\th{\theta}
\def\a{\alpha}
\def\om{\omega}
\def\Om{\Omega}
\def\v{\mathrm{v}}
\def\u{\mathbf{u}}
\def\w{\mathrm{w}}
\def\p{\partial}
\def\d{\delta}
\def\g{\gamma}
\def\l{\lambda}
\def\la{\langle}
\def\ra{\rangle}
\def\[{{\Big[}}
\def\]{{\Big]}}
\def\<{{\langle}}
\def\>{{\rangle}}
\def\({{\Big(}}
\def\){{\Big)}}
\def\bz{{\mathbf{z}}}
\def\by{{\mathbf{y}}}
\def\bx{{\mathbf{x}}}
\def\tr{\mathrm {tr}}
\def\W{{\mathcal W}}
\def\Ric{{\rm Ricci}}
\def\osc{{\rm osc}}
\def\Cap{\mbox{\rm Cap}}
\def\sgn{\mbox{\rm sgn}}
\def\mathcalV{{\mathcal V}}
\def\Law{{\mathord{{\rm Law}}}}
\def\dif{{\mathord{{\rm d}}}}
\def\dis{{\mathord{{\rm \bf d}}}}
\def\Hess{{\mathord{{\rm Hess}}}}
\def\min{{\mathord{{\rm min}}}}
\def\Vol{\mathord{{\rm Vol}}}
\def\bbbn{{\rm I\!N}}
\def\no{\nonumber}
\def\={&\!\!=\!\!&}
\def\bt{\begin{thm}}
\def\et{\end{thm}}
\def\bl{\begin{lem}}
\def\el{\end{lem}}
\def\br{\begin{remark}}
\def\er{\end{remark}}
\def\bd{\begin{definition}}
\def\ed{\end{definition}}
\def\bp{\begin{prp}}
\def\ep{\end{prp}}
\def\bc{\begin{cor}}
\def\ec{\end{cor}}
\def\bx{\begin{Examples}}
\def\ex{\end{Examples}}

\def\cA{{\mathcal A}}
\def\cB{{\mathcal B}}
\def\cC{{\mathcal C}}
\def\cD{{\mathcal D}}
\def\cE{{\mathcal E}}
\def\cF{{\mathcal F}}
\def\cG{{\mathcal G}}
\def\cH{{\mathcal H}}
\def\cI{{\mathcal I}}
\def\cJ{{\mathcal J}}
\def\cK{{\mathcal K}}
\def\cL{{\mathcal L}}
\def\cM{{\mathcal M}}
\def\cN{{\mathcal N}}
\def\cO{{\mathcal O}}
\def\cP{{\mathcal P}}
\def\cQ{{\mathcal Q}}
\def\cR{{\mathcal R}}
\def\cS{{\mathcal S}}
\def\cT{{\mathcal T}}
\def\cU{{\mathcal U}}
\def\cV{{\mathcal V}}
\def\cW{{\mathcal W}}
\def\cX{{\mathcal X}}
\def\cY{{\mathcal Y}}
\def\cZ{{\mathcal Z}}

\def\mA{{\mathbb A}}
\def\mB{{\mathbb B}}
\def\mC{{\mathbb C}}
\def\mD{{\mathbb D}}
\def\mE{{\mathbb E}}
\def\mF{{\mathbb F}}
\def\mG{{\mathbb G}}
\def\mH{{\mathbb H}}
\def\mI{{\mathbb I}}
\def\mJ{{\mathbb J}}
\def\mK{{\mathbb K}}
\def\mL{{\mathbb L}}
\def\mM{{\mathbb M}}
\def\mN{{\mathbb N}}
\def\mO{{\mathbb O}}
\def\mP{{\mathbb P}}
\def\mQ{{\mathbb Q}}
\def\mR{{\mathbb R}}
\def\mS{{\mathbb S}}
\def\mT{{\mathbb T}}
\def\mU{{\mathbb U}}
\def\mV{{\mathbb V}}
\def\mW{{\mathbb W}}
\def\mX{{\mathbb X}}
\def\mY{{\mathbb Y}}
\def\mZ{{\mathbb Z}}

\def\bB{{\mathbf B}}
\def\bP{{\mathbf P}}
\def\bQ{{\mathbf Q}}
\def\bE{{\mathbf E}}
\def\1{{\mathbf{1}}}

\def\sA{{\mathscr A}}
\def\sB{{\mathscr B}}
\def\sC{{\mathscr C}}
\def\sD{{\mathscr D}}
\def\sE{{\mathscr E}}
\def\sF{{\mathscr F}}
\def\sG{{\mathscr G}}
\def\sH{{\mathscr H}}
\def\sI{{\mathscr I}}
\def\sJ{{\mathscr J}}
\def\sK{{\mathscr K}}
\def\sL{{\mathscr L}}
\def\sM{{\mathscr M}}
\def\sN{{\mathscr N}}
\def\sO{{\mathscr O}}
\def\sP{{\mathscr P}}
\def\sQ{{\mathscr Q}}
\def\sR{{\mathscr R}}
\def\sS{{\mathscr S}}
\def\sT{{\mathscr T}}
\def\sU{{\mathscr U}}
\def\sV{{\mathscr V}}
\def\sW{{\mathscr W}}
\def\sX{{\mathscr X}}
\def\sY{{\mathscr Y}}
\def\sZ{{\mathscr Z}}\def\E{\mathbb E}

\def\fM{{\mathfrak M}}
\def\fA{{\mathfrak A}}

\def\geq{\geqslant}
\def\leq{\leqslant}
\def\ge{\geqslant}
\def\le{\leqslant}

\def\c{\mathord{{\bf c}}}
\def\div{\mathord{{\rm div}}}
\def\iint{\int\!\!\!\int}

\def\sb{{\mathfrak b}}

\def\bH{{\mathbf H}}
\def\bW{{\mathbf W}}
\def\bP{{\mathbf P}}
\def\bA{{\mathbf A}}\def\si{\sigma}
\def\bT{{\mathbf T}} \def\R{\mathbb R}\def\nn{\nabla}
\def\ff{\frac} \def\R{\mathbb R}  \def\ff{\frac} \def\ss{\sqrt} \def\B{\mathscr B}
\def\N{\mathbb N} \def\kk{\kappa} \def\m{{\bf m}}
\def\dd{\delta} \def\DD{\Dd} \def\vv{\varepsilon} \def\rr{\rho}
\def\<{\langle} \def\>{\rangle} \def\GG{\Gamma} \def\gg{\gamma}
  \def\nn{\nabla} \def\pp{\partial} \def\EE{\scr E}
\def\d{\text{\rm{d}}} \def\bb{\beta} \def\aa{\alpha} \def\D{\scr D}
  \def\si{\sigma} \def\ess{\text{\rm{ess}}}
\def\beg{\begin} \def\beq{\begin{equation}}  \def\F{\scr F}
\def\Ric{\text{\rm{Ric}}} \def\Hess{\text{\rm{Hess}}}
\def\e{\text{\rm{e}}} \def\ua{\underline a} \def\OO{\Omega}  \def\oo{\omega}
 \def\tt{\tilde} \def\Ric{\text{\rm{Ric}}}
\def\cut{\text{\rm{cut}}} \def\P{\mathbb P} \def\ifn{I_n(f^{\bigotimes n})}
\def\C{\mathscr C}      \def\aaa{\mathbf{r}}     \def\r{r}
\def\gap{\text{\rm{gap}}} \def\prr{\pi_{{\bf m},\varrho}}  \def\r{\mathbf r}
\def\Z{\mathbb Z} \def\vrr{\varrho} \def\l{\lambda}\def\ppp{\preceq}
\def\L{\mathscr L}\def\Tt{\tt} \def\TT{\tt}\def\II{\mathbb I}\def\ll{\lambda} \def\LL{\Lambda}
\def\b{\mathbf b}\def\u{{\mathbf u}}\def\E{\mathbb E} \def\BB{\mathbb B}\def\Q{\mathbb Q} \def\M{\scr M}
\def\H{\mathbb H}\def\DD{\Delta} \def\QQ{\mathbf Q}

\allowdisplaybreaks
\maketitle
  \begin{abstract} By using the ultracontractivity of a reference diffusion semigroup, Krylov's estimate is established for a class of degenerate SDEs with singular drifts, which leads to existence and pathwise uniqueness by means of Zvonkin's transformation. The main result is applied to singular SDEs on generalized Heisenberg groups.
\end{abstract} \noindent

\noindent
 {\bf AMS subject Classification:}\ 60H15,  35R60.   \\
\noindent
{\bf Keywords:}
Degenerate SDEs,  Krylov estimate, Zvonkin transformation, Heisenberg group.

\rm

\section{Introduction}

Since 1974 when Zvonkin \cite{Zv} proved the well-posedness of the Brownian motion  with bounded drifts, his argument (known as Zvonkin's transformation) has been developed  for more general models with singular drifts, see
  \cite{Ve, Kr-Ro, Zh0, Xi-Zh} and references within for non-degenerate SDEs, and \cite{AB1}-\cite{AB4} and \cite{Fe-Fl, W} for non-degenerate semilinear SPDEs.  In these references only Gaussian noise is considered, see also \cite{Pr2, XZhang2} for extensions to the case with jump.

In recent years, Zvonkin's transformation has been applied  in \cite{CDR, Pr, Wa-Zh, WZ16, Zh4} to a class of degenerate SDEs/SPDEs with singular drifts.  This type degenerate stochastic systems are called  stochastic Hamiltonian systems in probability theory. Consider, for instance, the following SDE for $(X_t,Y_t)$ on $\R^{2d}$ ($d\ge 1$):
\beq\label{E01} \beg{cases} \d X_t= Y_t\d t,\\\d Y_t = b_t(X_t, Y_t)\d t +\si_t(X_t,Y_t)\d W_t,\end{cases}\end{equation}where
$W_t$ is the $d$-dimensional Brownian motion, and $$b: [0,\infty)\times \R^{2d} \to \R^d,\ \ \si: [0,\infty)\times\R^{2d}\to \R^d\otimes \R^d$$are measurable. According to \cite[Theorem 1.1]{Zh4}, if there exists a constant $K>1$ such that $$ K^{-1}|v|\le |\si_t v|\le K|v|, \ \ t\ge 0, v\in\R^d,$$ and for some constant $p>2(1+2d)$,
 $$\sup_{t\ge 0} \|\nn \si_t\|_{L^p(\R^{2d})} + \int_0^\infty \|(1-\DD_x)^{\ff 1 3} b_t\|^p_{L^p(\R^{2d})}\d t<\infty,$$ then the SDE \eqref{E01} has a unique strong solution for any initial points. By a standard truncation argument,  the existence and pathwise uniqueness   up to the life time hold under the corresponding local conditions.

In this paper, we aim to extend this result to general degenerate SDEs, in particular, for singular diffusions on generalized Heisenberg groups.
As  typical models of hypoelliptic systems, smooth SDEs on Heisenberg groups have  been intensively investigated,  see for instance \cite{BG,BBB,BB,BGM,Li,W9} and references within for the study of
  functional inequalities, gradient estimates, Harnack inequalities, and   Riesz transforms. We will use these results to establish Krylov's estimates for singular SDEs and to prove the existence and uniqueness of strong solutions using Zvonkin's transformation.

In Section 2, we present a general result (see Theorem 2.1(3)) for the existence and uniqueness of degenerate SDEs with singular drifts, and then apply this result in Section 3 to singular diffusions on generalized Heisenberg groups.

\section{General results }

 For fixed constant $T>0$, consider the following SDE on $\R^N$:
 \beq\label{E0} \d X_t = Z_t(X_t) \d t +\si_t(X_t)\d B_t,\ \ t\in [0,T],\end{equation}
 where $B_t$ is the $m$-dimensional Brownian motion with respect to a complete filtered  probability space $(\OO, \F,\{\F_t\}_{t\in [0,T]}, \P)$, and
 $$Z: [0,T]\times \R^N\to \R^N,\ \ \si: [0,T]\times \R^N\to \R^N\otimes\R^m$$ are measurable and locally bounded. We are in particular interested in the case that $m<N$ such that this SDE is degenerate.

 Throughout the paper, we assume that for any $x\in \R^N$ and $s\in [0,T)$, this SDE has a unique solution $(X_{s,t}^x)_{t\in [s,T]}$ with $X_{s,s}^x=x$; i.e. it is a continuous adapted process such that
 $$X_{s,t}^x= x +\int_s^t Z_r(X_{s,r}^x)\d r +\int_s^t \si_r(X_{s,r}^x)\d B_r,\ \ t\in [s,T].$$
 Let $(P_{s,t})_{0\le s\le t\le T}$ be the associated Markov semigroup. We have
 $$P_{s,t} f(x) = \mathbb E f(X_{s,t}^x),\ \ f\in \B_b(\R^N)\cup \B^+(\R^N),\ \ x\in \R^N, 0\le s\le t\le T,$$
 where $\B_b$ (resp. $\B^+$) denotes the set of bounded (resp. nonnegative) measurable functions.
 The infinitesimal generator of the solution is
 $$\L_s:= \ff 1 2\sum_{i,j=1}^N (\si_s\si_s^*)_{ij}\, \pp_i\pp_j + \sum_{i=1}^N (Z_s)_i\, \pp_i.$$

 Now, let $\b: [0,T]\times \R^N\to \R^m$ be measurable. We intend to find reasonable conditions on $\b$ such that the following perturbed SDE is well-posed:
  \beq\label{E1} \d \tt X_t = \{Z_t+\si_t \b_t\}(\tt X_t) \d t +\si_t(\tt X_t)\d B_t,\ \ t\in [0,T].\end{equation}

 To state the main result, we introduce two spaces $L^q_{p,loc}([0,T]\times \R^N)$ and $W^q_{p,loc}([0,T]\times \R^N)$, for $p, q\ge 1$.
A real measurable  function $f$ defined on $[0,T]\times \R^N$ is said in $L^q_p([0,T]\times \R^N)$, if
$$\|f\|_{L^q_p}:= \bigg(\int_0^T \|f_t\|_{L^p(\R^N)}^q \d t\bigg)^{\ff 1 q}<\infty.$$
Next, if $f\in L^q_p([0,T]\times \R^N)$ such that $\nn f_s$ exists in weak sense for a.e. $s\in [0,T]$ and $|\nn f|\in L^q_p([0,T]\times \R^N)$, where $\nn$ is the gradient operator on $\R^N$, we write $f\in W^q_p([0,T]\times \R^N)$. Consequently, we write
$$f\in L^q_{p,loc}([0,T]\times \R^N)$$  if $h f\in L^q_p([0,T]\times \R^N)$ for any $h\in C_0^\infty(\R^N)$, and
 $$f\in W^q_{p,loc}([0,T]\times \R^N)$$ if $h f\in W^q_p([0,T]\times \R^N)$ for any $h\in C_0^\infty(\R^N)$. Moreover, a vector-valued function is in one of these spaces if so are its components.

Finally, a real function $f$ on $  \R^N$ is called $\si_t$-differentiable, if for any $v\in \R^m$ it is differentiable along the direction $\si_tv$; i.e.
 $$\nn_{\si_t v} f (x):= \ff{\d}{\d r} f(x+ r \si_t(x)v)\Big|_{r=0} $$ exists for any $x\in \R^N$.
 A real function $f$ on $ [0,T]\times \R^N$ is called $\si$-differentiable if $f_t$ is $\si_t$-differentiable for every $t\in [0,T].$ In this case, $\nn_\si f: [0,T]\times \R^N\to \R^m$ is defined by
 $$\<(\nn_\si f)_t(x), v\>:= \nn_{\si_t v}f_t (x),\ \ v\in \R^m, t\in [0,T], x\in\R^N.$$ When $\nn f$ exists, we have $\nn_\si f= \si^*\nn f$.
Let $$\BB=\big\{f\in C_b([0,T]\times\R^N): \nn_\si f\in C_b([0,T]\times\R^N; \R^m)\big\}.$$ Then $\BB$ is a Banach space with
$$\|f\|_\BB:= \|f\|_\infty+\|\nn_\si f\|_\infty,$$ where $\|\cdot\|_\infty$ is the uniform norm.
An $\R^m$-valued function    $g$ is said in the space $\BB^m$, if its  components belong  to $\BB$. Let
$$\|g\|_{\BB^m} =\|g\|_\infty+\|\nn_\si g\|_\infty,\ \ g\in \BB^m.$$
We make the following assumptions.

  \beg{enumerate} \item[$(A_1)$] $\si_t(x)$ is locally bounded in $(t,x)\in [0,T]\times\R^N$, and for any $R>0$, there exists a constant $c>0$ such that $$|\si_t(x) \b_t(x)|\ge c |\b_t(x)|,\ \ t\in [0,T], |x|\le R.$$
  \item[$(A_2)$] For any $f\in C_0^\infty([0,T]\times \R^N)$ and $\ll\ge 0$, the function
 \beq\label{QL}(Q_\ll f)_s(x):= \int_s^T \e^{-\ll(t-s)} P_{s,t} f_t(x)\d t,\ \ s\in [0,T], x\in \R^N\end{equation} satisfies that $\pp_s (Q_\ll f)_s, \scr L_s(Q_\ll f)_s$ exist and are locally bounded on $[0,T]\times \R^N$ with
 \beq\label{IT2}(\pp_s+\scr L_s -\ll)(Q_\ll f)_s +f_s=0.\end{equation}
\end{enumerate}
Assumption $(A_1)$ holds provided $\si$ is continuous and   has rank $m$.   Assumption $(A_2)$ holds if $Z_t(x)$ and $\si_t(x)$  are regular enough in $x$, for instance,
$C^2$-smooth in $x$ uniformly in $t\in [0,T]$.

For any $p,q\ge 1$, let $\|\cdot\|_{p\to q}$ denote the operator norm from $L^p(\R^N;\d x)$ to $L^q(\R^N;\d x)$.
To introduce the integrability conditions for the drift $\b$, we need the following two classes of pairs $(p,q)\in (1,\infty]^2$:
\beg{align*}\scr K_1:=\big\{(p,q)\in   (1,\infty]^2:\ &\text{there\ exists\  }\gg\in L^{\ff q {q-1}}([0,T]) \text{\ such\ that} \\ & \ \|P_{s,t}\|_{p\to\infty}\le \gg(t-s),\ \ 0\le s< t\le T\big\},\\
 \scr K_2:=\big\{(p,q)\in  (1,\infty]^2:\ &\text{there\ exists\  }\gg\in L^{\ff q {q-1}}([0,T]) \text{\ such\ that} \\
 &\ \ \|\si_t^*\nn P_{s,t}\|_{p\to\infty}\le \gg(t-s),\ \ 0\le s< t\le T\big\}.\end{align*}
 Obviously, both $\scr K_1$ and $\scr K_2$ are increasing sets; that is, if $(p,q)\in \scr K_i$ then $(p',q')\in \scr K_i$ for $p'\ge p$ and $q'\ge q, i=1,2.$  We will also use the following class
 $$2\scr K_1:=\{(2p,2q):(p,q)\in\scr K_1\}.$$
Clearly, $2\scr K_1\subset\scr K_1$. 
When $Z=0$ and $\si\si^*=I_{N\times N}$, the $N\times N$-identity matrix,  we have $P_{s,t}=P_{t-s}$ for the standard heat semigroup $P_t$, so that
\beq\label{GRD} \|P_{s,t}\|_{p\to\infty}\le C(t-s)^{-d/(2p)},\ \ \|\si^*\nn P_{s,t}\|_{p\to\infty}\le C(t-s)^{-\ff 1 2- d/(2p)},\ \ 0\le s<t\le T\end{equation}
for some constant $C>0$, then
\beq\label{KKW}scr K_1\textcolor{red}{\supset}\Big\{(p,q)\in (1,\infty]^2:\ \ff 1 q + \ff d {2p} <1\Big\},\ \  \scr K_2 \supset\Big\{(p,q)\in (1,\infty]^2:\ \ff 2 q + \ff d p <1\Big\}.\end{equation}
These formulas  also hold for elliptic diffusions satisfying \eqref{GRD}. But for degenerate diffusions the dimension $d$ in this display will be enlarged, see for instance the proof of Theorem \ref{T3.1} below.

We are now ready to state the main result in this section. In particular, the first  assertion implies that if $\scr K_1\cap\scr K_2\ne \emptyset$,     then for any $\b\in C^\infty_0([0,T]\times\mathbb{R}^{N})$ and large enough $\ll>0$, the equation
\beq\label{IE} \u_s = \int_s^T \e^{-\ll(t-s)} P_{s,t} \big\{\nn_{\si_t \b_t} \u_t +\si_t \b_t\big\}\d t,\ \ s\in [0,T]\end{equation} has a unique solution $\u=:\Xi_\ll\b\in   \mathbb B^m$. We write $f\in C^{1,2}([0,T]\times\R^N)$ if $f$ is a function on $[0,T]\times \R^N$ such that $\pp_t f_t(x)$ and $ \nn^2 f_t(x)$ exist and continuous in $(t,x)$.

\beg{thm}\label{T2.0}  Assume $(A_1)$ and  $(A_2)$. Then the the following assertions hold.
\beg{enumerate} \item[$(1)$] For any $(p,q)\in \scr K_1\cap\scr K_2$ and $L>0$, there exists $\tt\ll>0$ such that for any $\ll\ge \tt\ll$, and  $\b: [0,T]\times\R^N\to \R^m$ with  $|\si\b|+|\b|\in L^q_p([0,T]\times \R^N)$ and $\||\si\b|\|_{L^q_p}\lor \||\b|\|_{L^q_p}\le L$,   the equation \eqref{IE}
has a unique solution $\u=:\Xi_\ll\b$ in $\BB^m$. Moreover,   there exists a decreasing function $\psi: [\tilde{\ll},\infty)\to (0,\infty)$ with $\psi(\infty):=\lim_{\lambda\to\infty}\psi(\lambda)=0$ such that for any $\b,\tt\b: [0,T]\times\R^N\to \R^m$ with  $ \||\si\b|\|_{L^q_p}, \||\si\tt \b|\|_{L^q_p},\||\b|\|_{L^q_p}, \||\tt \b|\|_{L^q_p}\le L,$
\beq\label{QN} \|\Xi_\ll \b -\Xi_\ll\tt \b\|_{\BB^m} \le \psi(\ll) \||\b-\tt\b|+|\si(\b-\tt \b)|\|_{L^q_p},\ \ \ll\ge \tt\ll.\end{equation}
\item[$(2)$]  If $|\mathbf{b}|\in L^q_p([0,T]\times \R^N)$ for some $p,q\ge 1$ with $(p,q)\in 2 \scr K_1$, then for any  $x\in \R^N$ the SDE $\eqref{E1}$ has a weak solution $(\tt X_t)_{t\in [0,T]}$ starting at $x$ with respect to a probability $\Q$ such that   $\E_\Q\e^{\ll \int_0^T|\b_t(\tt X_t)|^2\d t}<\infty$ holds for all $\ll>0$.
\item[$(3)$]  Assume further that 
 \begin{enumerate}
 \item[(i)] for large enough $\ll>0$, $\Xi_\ll f\in C^{1,2}([0,T]\times\R^N)$ holds for any $f\in C^\infty_0([0,T]\times\mathbb{R}^{N})$;  
 \item[(ii)] there exists $(p,q)\in 2\scr K_1\cap \scr K_2$ such that  $|\b|+|\nabla \sigma|\in L^q_{p,loc}([0,T]\times \R^N)$, $Z\in W^{q}_{p,loc} ([0,T]\times \R^N)$,  and  for large enough $\ll>0$  there hold
\beq\label{QN1}\nn_\si \Xi_\ll (h\b)\in W^{q}_{p,loc}  ([0,T]\times \R^N),\ \  h\in C_0^\infty(\R^N),\end{equation}
  \beq\label{QN2} \lim_{\ll\to\infty} \|h\nn\Xi_\ll (h\b)\|_\infty=0,\ \ h\in C_0^\infty(\R^N).\end{equation} 
  \end{enumerate} 
  Then for any $x\in \R^N$  the SDE $\eqref{E1}$  has a unique solution $\tt X_t$ starting at $x$  up to the life time $\zeta:= \lim_{n\to\infty} T\land \inf\{t\in [0,T]: |\tt X_t|\ge n\}$. \end{enumerate}
    \end{thm}
By \eqref{IE} and the definition of $Q_\ll$ in $(A_2)$, we have
\beq\label{QW} \Xi_\ll\b= Q_\ll\big\{\nn_{\si\b}\Xi_\ll \b +\si\b\big\}.\end{equation}
If  \eqref{KKW}  holds, Theorem \ref{T2.0} (3) ensures the strong well-posedness when $|\b|\in L_p^q$ for some $p,q\ge 1$ with $\ff d p + \ff 2 q<1$, which coincides known optimal result in the elliptic setting.  In  the elliptic case there exists much weaker sufficient conditions for  the well-posedness,  for instance, in a recent paper by  Xicheng Zhang and Guohua Zhao \cite{ZZ},  the drift is allowed to be distributions (not necessarily functionals). 

To prove Theorem \ref{T2.0}, we first investigate the Krylov estimate  and the weak existence  for \eqref{E1}.

\subsection{Krylov's estimate and weak existence}

 \beg{thm}\label{T2.1}  Assume $(A_1)$ and $(A_2)$.  Let   $p,q\ge 1$ such that $(p,q)\in \scr K_1$ and $|\mathbf{b}|^2 \in L^q_p([0,T]\times \R^N)$.
 \beg{enumerate} \item[$(1)$]  For any $(p',q')\in \scr K_1$ there exists a constant $\kk>0$ such that for any $s\in [0, T)$ and solution $(\tt X_{s,t} )_{t\in [s,T]}$ of $\eqref{E1}$ from time $s$ with $\int_s^T|\b_t(\tt X_{s,t})|^2\d t <\infty$,
 \beq\label{Q1} \E\bigg( \int_s^T |f_t(\tt X_{s,t})|\d t \bigg|\F_s\bigg)\le\kk  \|f\|_{L^{q'}_{p'}},\ \ f\in L^{q'}_{p'}([0,T]\times\R^N), s\in [0,T].\end{equation} Consequently, for any $f\in L^{q'}_{p'}([0,T]\times\R^N)$ with $(p',q')\in\scr K_1$  and any $\ll>0$, there exists a constant $c(f,\ll)\in (0,\infty)$ such that
 \beq\label{QE} \E\Big(\e^{\ll \int_s^T|f_t(\tt X_{s,t})|\d t}\Big|\F_s\Big)\le c(f,\ll),\ \ s\in [0,T].\end{equation}
 \item[$(2)$] The assertion in Theorem \ref{T2.0}(2) holds.\end{enumerate} \end{thm}

To prove this result, we need the following lemma.
\beg{lem}\label{L2.1}  Assume $(A_1)$  and  $(A_2)$.
\beg{enumerate} \item[$(1)$] For any $(p,q)\in \scr K_1$, there exists a decreasing function $\psi: [0,\infty)\to (0,\infty)$ with $\psi(\infty):=\lim_{\ll\t\infty}\psi(\ll)=0$ such that for any $\ll\ge 0$,  $(Q_\ll, C_0^\infty([0,T]\times\R^N))$ in $(A_2)$ extends uniquely to a bounded linear operator $ Q_\ll: L_p^q([0,T]\times\R^N)\to L^\infty([0,T]\times\R^N)$ with
  $$\|Q_\ll f\|_\infty \le \psi(\ll) \|f\|_{L^q_p},\ \ f\in L_p^q([0,T]\times\R^N),\ \ll\ge 0.$$
 \item[$(2)$] For any  $(p,q)\in\scr K_1\cap\scr K_2$,   $(Q_\ll, C_0^\infty([0,T]\times\R^N))$ extends to a unique bounded linear  operator $Q_\ll:  L^q_p([0,T]\times\R^N)\to \BB $  such that
$$\|Q_\ll f\|_\mathbb{B}:= \|Q_\ll f\|_\infty + \|\nn_\si Q_\ll f\|_\infty\le \psi(\ll) \|f\|_{L^q_p}, \ \ \ll\ge 0, f\in L_p^q([0,T]\times\R^N) $$ holds for some decreasing   $\psi: [0,\infty)\to (0,\infty)$ with $\psi(\infty)=0$. \end{enumerate} \end{lem}

\beg{proof} We only prove (1) since that of (2) is completely similar. For any $(p,q)\in \scr K_1$, there exists $\gg\in L^{\ff q {q-1}}([0,T])$ such that for any $\ll\ge 0$ and $f\in C_0^\infty([0,T]\times \R^N)$,
\beg{align*} \|Q_\ll f\|_\infty &  \le \sup_{s\in [0,T]} \int_s^T \e^{-\ll(t-s)}  \gg(t-s) \|f_t\|_{L^p(\R^N)} \d t \\
&\le \sup_{s\in [0,T]} \bigg(\int_s^T \big|\e^{-\ll(t-s)} \gg (t-s) \big|^{\ff q {q-1}} \d t\bigg)^{\ff{q-1}q}  \|f\|_{L^q_p},\ \ \ll\ge 0.\end{align*} So, assertion (1) holds with
$$\psi(\ll):= \bigg(\int_0^T \big|\e^{-\ll t}  \gg(t) \big|^{\ff q {q-1}} \d t\bigg)^{\ff{q-1}q}.$$  \end{proof}

We will also need the following  lemma which reduces the desired Krylov's estimate  to $f\in C_0^\infty([0,T]\times \R^N).$  It can be proved using a standard approximation argument.

\beg{lem}\label{L2.2} Let $s\in [0,T]$ and $p,q\ge 1$ For any two stopping times $\tau_1\le \tau_2$, measurable process $(\xi_t)_{t\in [s,T]}$ on $\R^N$, and
random variable  $\eta\ge 0$, if  the inequality
\beq\label{QD} \E\bigg(\int_{s\land \tau_1}^{T\land\tau_2} f_t(\xi_t)\d t \bigg|\F_s\bigg)\le \|f\|_{L^q_p} \E(\eta|\F_s)\end{equation}
holds for all nonnegative $f\in C_0^\infty([0,T]\times \R^N)$, it holds for all nonnegative $f\in L^q_p([0,T]\times \R^N).$\end{lem}

\beg{proof}[Proof of Theorem \ref{T2.1}] (1) According to the Khasminskii estimate, see
\cite[Lemma 5.3]{Zh0}, \eqref{QE} follows from \eqref{Q1}. For simplicity, we only prove for $s=0$. To prove \eqref{Q1}, we first consider  $\b=0$.
Let $(X_t)_{t\in [0,T]}$  solve \eqref{E0} and let
$$\tau_n=\inf\{t\in [0,T]:\ |X_t|\ge n\},\ \ n\ge 1.$$
 For  $0\le f\in C_0^\infty([0,T]\times \R^N)$, take $u^{(\ll)}=Q_\ll f$ for $\ll\ge 0$. By $(A_2)$ and  It\^o's formula, we obtain
 \beg{align*}&0\le \E \big(u_{T\land \tau_n}^{(\ll)}(X_{T\land\tau_n})\big|\F_s\big) \\
& = u_{0}^{(\ll)}(X_{s\land\tau_n})+ \E\bigg(\int_{0}^{T\land\tau_n} (\pp_t +\scr L_t )
u_t^{(\ll)}(X_t)\d t\bigg|\F_s\bigg) \\
&\le  (1+\ll T)\|u^{(\ll)}\|_\infty -\E\bigg(\int_{0}^{T\land\tau_n}  - f_t(X_t)\d t\bigg|\F_s\bigg).\end{align*}
Noting that $u^{(\ll)}= Q_\ll f$, combining this with Lemma  \ref{L2.1} and Lemma \ref{L2.2}, for any $(p',q')\in\scr K_1$ there exists decreasing $\psi: [0,\infty)\to (0,\infty)$ with $\psi(\infty)=0$ such that
 \beq\label{Q4} \beg{split} &\E\bigg(\int_{0}^{T\land\tau_n}  f_t(X_t)\d t\bigg|\F_0\bigg)\le \psi(\ll) \|f\|_{L^{q'}_{p'}} (1+\ll T),\  \ 0\le f\in L^{q'}_{p'}([0,T]\times\R^N).\end{split}\end{equation} Letting $n\to\infty$ we prove  \eqref{Q1}   for $\b=0$.
 In general, let $(\tt X_{t})_{t\in [0,T]}$ solve \eqref{E1}  with $\int_0^T|\b_t(\tilde{X}_{t})|^2\d t<\infty$. Define
 $$T_n=\inf\bigg\{t\in [0,T]: \int_0^t |\b_r(\tt X_{r})|^2\d r\ge n\bigg\},\ \ n\ge 0,$$ where $\inf\emptyset:=\infty$ by convention.  Let
 \beg{align*} & R_{n}= \exp\bigg[-\int_0^{T\land T_n} \big\<\b_r(\tilde{X}_{r}), \d B_r\big\>- \ff 1 2 \int_0^{T\land T_n} | \b_r(\tilde{X}_{r})|^2\d r\bigg],\\
 &\tt B_t= B_t + \int_0^{t\land T_n} \b_r(\tilde{X}_{r})\d r,\ \ t\in [s,T].\end{align*}
  Then  under the probability $R_n\P$, $(\tilde{X}_{r}, \tt B_r)_{r\in [0, T\land T_n]}$ is a weak solution to the SDE \eqref{E1} for $\b=0$. So, by the assertion for $\b=0$, there exists a constat $c>0$ such that
 \beq\label{APP3'} \E\bigg[R_{n} \bigg(\int_0^{T\land T_n} f(r, \tt X_{r})\d r\bigg)^2\bigg|\F_0\bigg] \le c \|f\|^2_{L_{p'}^{q'}(T)}.\end{equation}
 By H\"{o}lder inequality and \eqref{Q1} for $\b=0$, there exists a constant $c'>0$  such that
 \beg{align*} &\E \big(R_{n}^{-1}\big|\F_0\big)= \E \big(R_{n} \e^{2\int_0^{T\land T_n} \<  \b_r(\tt X_{r}),\d B_r\>+  \int_0^{T\land T_n} | \b_r(\tt X_{r})|^2\d r}\big|\F_0\big)\\
  &\le \ss{\E \big(R_{n} \e^{4\int_0^{T\land T_n} \<\b_r(\tt X_{r}),\d \tt B_r\>- 8 \int_0^{T\land T_n} | \b_r(\tt X_{r})|^2\d r}\big|\F_0\big)} \\
 &\qquad\times  \ss{ \E \big(R_{n} \e^{ 6 \int_0^{T\land T_n} |\b_r(\tt X_{r})|^2\d r}\big|\F_0\big)}\\
 &= \ss{ \E \big(R_{n} \e^{ 6 \int_0^{T\land T_n} |\b_r(\tt X_{r})|^2\d r}\big|\F_0\big)}\le c',\end{align*}where the last step follows from
 $|\b|^2\in L^{q/2}_{p/2}([0,T]\times \R^N)$ for some $(p,q)\in2\scr K_1$ and \eqref{QE} for $\b=0$.
 Then there exists a constant $C>0$ such that
 \beg{align*} &\bigg\{\E\bigg(\int_0^{T\land T_n} f_r(\tt X_{r})\d r\Big|\F_0\bigg)\bigg\}^2 \\
 &\le \E\bigg[R_{n} \bigg(\int_0^{T\land T_n} f_r(\tt X_{r})\d r\bigg)^2\Big|\F_0\bigg]\cdot \E \big(R_{n}^{-1}\big|\F_0\big)\le C \|f\|^2_{L_{p'}^{q'}(T)}. \end{align*}
By letting $n\to\infty$ we prove \eqref{Q1}.

 (2) Assume that $|\b|\in L^q_p([0,T]\times\R^N)$ for some $(p,q)\in  2\scr K_1$.
 Let $(X_t)_{t\in [0,T]}$ solve \eqref{E0} with $X_0=\tilde{X}_0$. We intend to show that it is a weak solution of \eqref{E1} under a weighted probability $\Q:= R\P$, where $R\ge 0$ is a probability density, and thus finish the proof. Since by \eqref{Q1} for $\b=0$ we have
\beq\label{*P} \E\int_0^T |\b_t(X_t)|^2 \d t \le \kk\||\b|^2\|_{L^{q/2}_{p/2}}=\kk\|\b\|^2_{L^{q}_{p}}<\infty \end{equation} for some constant $\kk>0$.
 Then
 $$T_n:= \inf\bigg\{s\in [0,T]: \int_0^s |\b_t|^2(X_t)\d t\ge n\bigg\}\uparrow\infty\ \text{as}\ n\uparrow\infty,$$ where we set $\inf\emptyset =\infty$ by convention.
 For any $n\ge 1$, let
  $$R_n =\exp\bigg[\int_0^{T\land T_n}\<\b_t(X_t), \d B_t\> -\ff 1 2 \int_0^{T\land T_n} |\b_t(X_t)|^2 \d t\bigg].$$
   By Girsanov's theorem,    $\{R_n\}_{n\ge 1}$ is a martingale and   $\Q_n:= R_n\P$ is a probability measure such that
  $$\tt B_t:= B_t- \int_0^{t\land T_n} \b_s(X_s)\d s,\ \ t\in [0, T]$$ is an $m$-dimensional Brownian motion.  Rewriting \eqref{E0} by
 \beq\label{Q6a}\d X_t= (Z_t+ \si_t\b_t)(X_t) \d t + \si_t (X_t)\d \tt B_t,\ \ t\in [0,T\land T_n],\end{equation} we see that $(X_t, \tt B_t)_{t\in [0,T\land T_n]}$ is a weak solution of \eqref{E1} up to time $T\land T_n$. To extend this solution to time $T$, it suffices to show that the martingale $(R_n)_{n\ge 1}$ is uniformly integrable, so that $R:= \lim_{n\to \infty} R_n$ is a probability density, and $(X_t, \tt B_t)_{t\in [0,T]}$ is a weak solution of \eqref{E1} under the probability $\Q:= R\P.$
  Therefore, it remains to prove
  \beq\label{Q5} \sup_{n\ge 1} \E [R_n\log R_n] <\infty,\ \ n\ge 1.\end{equation}
    Since $(\tt B_t)_{t\in [0,T]}$ is an $m$-dimensional Brownian motion under probability $\Q_n:= R_n\P$, by \eqref{Q6a} and Theorem \ref{T2.0}(1), for any $(p',q')\in\scr K_1$ there exists a constant $\kk>0$ such that
  $$\E_{\Q_n} \int_0^{T\land T_n} f_t(X_t)\d t \le \kk \|f\|_{L^{q'}_{p'}},\ \ 0\le f\in L^{q'}_{p'}([0,T]\times\R^N), n\ge 1.$$
  Applying this estimate to $f= |\b|^2$, we arrive at
  $$2 \E[R_n\log R_n]= \E_{\Q_n} \int_0^{T\land T_n} |\b_t|^2(X_t)\d t \le \kk \||\b|^2\|_{L^{q/2}_{p/2}},\ \ n\ge 1.$$
  This implies  \eqref{Q5} and $\E_\Q\int_0^T |\b_t|^2(X_t) \d t<\infty$. Then   the proof is finished since by \eqref{QE} for $f=|\b|^2$ in Theorem \ref{T2.1} (1), $\||\b|^2\|_{L^{q/2}_{p/2}}<\infty$ implies
  $ \E_\Q\e^{\ll \int_0^T |\b_t|^2(X_t) \d t}<\infty$ for all $\ll>0$.\end{proof}

\subsection{Proof of Theorem \ref{T2.0}}

Since Theorem \ref{T2.0}(2) follows from   Theorem \ref{T2.1}(2), we only prove Theorem \ref{T2.0}(1),(3).

\beg{proof}[Proof of Theorem \ref{T2.0}(1)] Let $\||\b|\|_{L^q_p}\le L$ for some $(p,q)\in\scr K_1\cap\scr K_2$. We first prove the existence and uniqueness of $\Xi_\ll \b$ for  large enough $\ll>0$. Consider the operator $\scr K_\ll$   on $\BB^m:$
$$\scr K_\ll \u := Q_\ll \big\{\nn_{\si\b}\u+ \si\b\big\},\  \ \u\in \BB^m.$$ By the fixed-point theorem, it suffices to show that $\scr K_\ll$ is contractive in $\BB^m$ for large enough $\ll>0$.

By Lemma \ref{L2.1}, for any $\u,\tt\u\in\BB^m$ we have
\beg{align*} &\|\scr K_\ll \u - \scr K_\ll \tt\u\|_{\BB^m}= \|Q_\ll\nn_{\si\b}(\u -\tt\u)\|_{\BB^m} \\
&\le \psi(\ll) \|\nn_\si (\u-\tt\u) \|_\infty \||\b|\|_{L^q_p}\le \psi(\ll) L\|\u-\tt\u\|_{\BB^m}.\end{align*}
Since $\psi(\ll)\to 0$ as $\ll\to\infty$, there exists $\ll_L>0$ such that $\psi(\ll_L)\le \ff 1 {2L}$. So, when $\ll\geq \ll_L,$    the map $\scr K_\ll$ is contractive in $\BB^m$. By the fixed point theorem, there exists a unique $\u\in \BB^m$ such that $\u=\scr K_\ll \u$, which is  denoted by $\Xi_\ll \b$.

Next, let $\||\b|\|_{L^q_p}, \||\tt \b|\|_{L^q_p}, \||\si\b|\|_{L^q_p}, \||\si\tt \b|\|_{L^q_p}\le L.$ By \eqref{QW},  Lemma \ref{L2.1} and $\psi(\ll)\le \ff 1 {2 L}$ for $\ll\ge \ll_L$, we have
\beg{align*} &\|\Xi_\ll \b- \Xi_\ll \tt\b\|_{\BB^m}   \le \psi(\ll) \big(\|\nn_{\si \b}\Xi_\ll \b- \nn_{\si \tt \b}\Xi_\ll \tt\b\|_{L^q_p}+\|\si(\b-\tt \b)\|_{L^q_p}\big)\\
&\le \psi(\ll) \big(\|\nn_{\si \b}(\Xi_\ll \b- \Xi_\ll\tt \b) +\nn_{\si(\b-\tt\b)} \Xi_\ll \tt\b\|_{L^q_p}   +\|\si(\b-\tt \b)\|_{L^q_p}\big)\\
& \le \psi(\ll) \big( \|\si(\b-\tt\b)\|_{L^q_p}+\|\Xi_\ll\tt \b\|_{\BB^m} \|\b-\tt\b\|_{L^q_p}\big)  +\psi(\ll) \||\b|\|_{L^q_p} \|\Xi_\ll\b-\Xi_\ll\tt\b\|_{\BB^m}\\
&\le \psi(\ll) \big( \|\si(\b-\tt\b)\|_{L^q_p}+\|\Xi_\ll\tt \b\|_{\BB^m} \|\b-\tt\b\|_{L^q_p}\big)  +\ff 1 2 \|\Xi_\ll\b-\Xi_\ll\tilde{\b}\|_{\BB^m},\ \ \ll\ge \ll_L.\end{align*}
Thus,
\beq\label{PPH} \|\Xi_\ll \b- \Xi_\ll \tt\b\|_{\BB^m}\le 2 \psi(\ll) \big( \|\si(\b-\tt\b)\|_{L^q_p}+\|\Xi_\ll\tt \b\|_{\BB^m} \|\b-\tt\b\|_{L^q_p}\big),\ \ \ll\ge \ll_L.\end{equation}
Applying this inequality to $\b=0$ we obtain $\|\Xi_\ll \tt\b\|_{\BB^m}\le 2\psi(\ll)\|\si\tt\b\|_{L^q_p}\le 1,$ so that \eqref{PPH} gives
$$\|\Xi_\ll \b- \Xi_\ll \tt\b\|_{\BB^m}\le 2 \psi(\ll)  \big(\||\b-\tt\b|\|_{L^q_p}+\|\si(\b-\tt\b)\|_{L^q_p}\big),\ \ \ll\ge \ll_L.$$
Then the   proof is finished.
\end{proof}

 To prove Theorem \ref{T2.0}(3), we consider Zvonkin's transformation
 \beq\label{Q6} \theta_t^{(\ll)}(x):= x+ (\Xi_\ll \b)_t(x),\ \ x\in\R^N, t\in [0,T] \end{equation} for large enough $\ll>0$.
 We have the following result.

 \beg{lem}\label{L2.3} Assume $(A_1)$-$(A_2)$ and Theorem \ref{T2.0}(3)(i). If $|\b|\in L^q_{p,loc}([0,T]\times\R^N)$ for some $(p,q)\in\scr K_1\cap\scr K_2$, then for large enough $\ll>0$, any solution $(\tt X_t)_{t\in [0,T]}$ of the SDE $\eqref{E1}$, any $k\geq 1$, and $h_k\in C_0^\infty(\R^N)$ such that $h_k|_{B(0,k)}=1$,
 \beq\label{E2}\d \theta^{(\ll,k)}_t(\tt X_t)= \big\{Z_t(\tt X_t)+\ll (\Xi_\ll h_k\b)_t(\tt X_t)\big\} \d t + \nn_{\si_t(\tt X_t)\d B_t}\theta_t^{(\ll,k)}(\tt X_t),\ \ t\in [0,T\wedge\tt\tau_k],\end{equation}
where
  $$\tt\tau_k:= \inf\{t\in [0,T]: |\tt X_t|\ge k\},$$
and
\begin{equation*}\begin{split}
&\theta_t^{(\ll,k)}(x):= x+ (\Xi_\ll h_k\b)_t(x),
\ \ x\in\R^N, t\in [0,T].\end{split}\end{equation*}
 \end{lem}
\beg{proof} When $\theta_t^{(\ll,k)}$ is second-order differentiable with bounded derivatives, the desired formula follows from \eqref{IT2},\eqref{QW} and  It\^o's formula.   In general, we use the following   approximation argument    as in \cite{Xi-Zh}.
Let $\{\b^{(n)}\}_{n\ge 1}\subset C_0^\infty([0,T]\times \R^N)$ such that
\beq\label{Q7} \lim_{n\to\infty}  \|h_k\b-h_k\b^{(n)}  \|_{L^q_p}=0.\end{equation}
Since $\sigma$ is locally bounded, we have
\beq\label{Q7'} \lim_{n\to\infty}  \|\sigma h_k\b-\sigma h_k\b^{(n)}  \|_{L^q_p}=0.\end{equation}
  Let $\theta^{(\ll,n,k)} $ be defined in \eqref{Q6} for $h_k\b^{(n)}$ replacing $\b$ respectively, i.e.
\beq\begin{split}\label{Q6'} &\theta_t^{(\ll,n,k)}(x):= x+ (\Xi_\ll h_k\b^{(n)})_t(x),
\ \ x\in\R^N, t\in [0,T], \ll\ge 0.\end{split}\end{equation}
  By $(A_3)$,   \eqref{QW}  and \eqref{Q6'}, we have
  \beg{align*} &(\pp_s +\scr L_s +\nn_{\si_s\b_s }) \theta_s^{(\ll,n,k)}\\
  & = Z_s+ \si_s \b_s +\nn_{\si_s\b_s}(\Xi_\ll h_k\b^{(n)})_s+\ll (\Xi_\ll h_k\b^{(n)})_s
  -\big\{\nn_{\si_s h_k\b_s^{(n)}} (\Xi_\ll h_k\b^{(n)})_s +\si_sh_k\b_s^{(n)}\big\}\\
  &= Z_s+\ll (\Xi_\ll h_k\b^{(n)})_s+\si_s (\b_s-h_k\b_s^{(n)}) + \nn_{\si_s(\b_s-h_k\b_s^{(n)})}( \Xi_\ll h_k\b^{(n)})_s.\end{align*}
  So, by \eqref{E1} and It\^o's formula, we have
  \beq\label{Q8} \beg{split} &\theta_{t\land \tt\tau_k}^{(\ll,n,k)}(\tt X_{t\land \tt\tau_k})- \theta_0^{(\ll,n,k)} (\tt X_0) \\
  &= \int_0^{t \land \tt\tau_k}\Big(Z_s+\ll (\Xi_\ll h_k\b^{(n)})_s+\si_s (\b_s-h_k\b_s^{(n)}) + \nn_{\si_s(\b_s-h_k\b_s^{(n)})}( \Xi_\ll h_k\b^{(n)})_s\Big)(\tt X_s)\d s \\
  &\qquad + \int_0^{t\land \tt\tau_k} \si_s(\tt X_s)\d B_s + \int_0^{t \land \tt\tau_k}\nn_{\si_s(\tt X_s) \d B_s} (\Xi_\ll h_k\b^{(n)})_s(\tt X_s),\ \ k\ge 1, t\in [0,T].\end{split} \end{equation}
By Theorem \ref{T2.0}(1) and \eqref{Q7}, for large enough $\ll>0$,
$$\lim_{n\to\infty} \|\theta^{(\ll,n,k)}-\theta^{(\ll,k)}\|_{\BB^m}=\lim_{n\to\infty} \|\Xi_\ll h_k\b-\Xi_\ll h_k\b^{(n)}\|_{\BB^m}=0.$$
Then
\beg{align*}&\lim_{n\to\infty} \big\{\theta_{t\land\tau_k}^{(\ll,n,k)}(\tt X_{t\land\tau_k})- \theta_0^{(\ll,n,k)} (\tt X_0)\big\}= \theta^{(\ll,k)}_{t\land\tau_k}(\tt X_{t\land \tau_k})- \theta^{(\ll,k)}_0(\tt X_0),\\
&\lim_{n\to\infty} \int_0^{T\wedge\tilde{\tau}_k}  |(\Xi_\ll h_k\b^{(n)}- \Xi_\ll h_k\b)_s|(\tt X_s)\d s =0.\end{align*}  Since $h_k|_{B(0,k)}=1$, combining these with   \eqref{Q1} and  the local boundedness of $\si$, we may find  out a constant $C>0$ such that
\beg{align*} &\lim_{n\to\infty} \E\int_0^{T\land \tt\tau_k} \Big(|\nabla_{\si_s} (\Xi_\ll h_k\b_s-\Xi_\ll h_k\b_s^{(n)})|^2\\
  &\qquad\qquad\qquad+ |\nn_{\si_s(\b_s-h_k\b_s^{(n)})}( \Xi_\ll h_k\b^{(n)})_s|+|\si_s(\b_s-h_k\b_s^{(n)})|\Big)(\tt X_s)\d s\\
&\le C \lim_{n\to\infty} \Big(\|\Xi_\ll h_k \b -\Xi_\ll h_k\b^{(n)})\|_{\mathbb{B}^m}^2+ \|h_k\b-h_k\b^{(n)}\|_{L^q_p}\Big)=0.\end{align*}
Therefore, letting $n\to\infty$ in \eqref{Q8}, we obtain
\beg{align*}&\theta_{t\land \tt\tau_k}^{(\ll,k)}(\tt X_{t\land \tt\tau_k})- \theta_0^{(\ll,k)} (\tt X_0)
   = \int_0^{t \land \tt\tau_k}\Big(Z_s+\ll (\Xi_\ll h_k \b)_s \Big)(\tt X_s)\d s \\
   &\quad +
 \int_0^{t\land \tt\tau_k} \si_s(\tt X_s)\d B_s + \int_0^{t \land \tt\tau_k}\nn_{\si_s(\tt X_s) \d B_s} (\Xi_\ll h_k\b )_s(\tt X_s),\ \ k\ge 1, t\in [0,T]. \end{align*}
 This means that \eqref{E2} holds for $t\le T\land\tt\tau_k.$
\end{proof}

By Lemma \ref{L2.3}, the uniqueness of the SDE  \eqref{E1} follows from that of \eqref{E2}.  As  in \cite{Kr-Ro, Zh0}, to prove the uniqueness of \eqref{E2}  we  will use the following result for the maximal operator: for any $N\geq 1$,
$$\M h(x):=\sup_{r>0}\frac{1}{|B(x,r)|}\int_{B(x,r)}h(y)\d y,\ \  h\in L^1_{loc}(\mathbb{R}^N), x\in \R^N,$$ where $B(x,r):=\{y: |x-y|<r\},$  see   \cite[Appendix A]{CD}.

 \begin{lem} \label{Hardy}  There exists a constant $C_N>0$ such that for any continuous and weak differentiable function $f$,
 \beq\label{HH1}
|f(x)-f(y)|\leq C_N|x-y|(\M |\nabla f|(x)+\M |\nabla f|(y)),\ \  {\rm a.e.}\ x,y\in\R^N.\end{equation}
Moreover, for any $p>1$,    there exists a constant $C_{N,p}>0$ such that
\beq\label{HH2}
\|\M f\|_{L^p}\leq C_{N,p}\|f\|_{L^p},\ \ f\in L^p(\R^N).
 \end{equation}
\end{lem}

\begin{proof}[Proof of Theorem \ref{T2.0}(3)] By Theorem \ref{T2.0}$(2)$ and  the Yamada-Watanabe principle, it suffices to prove the pathwise uniqueness. Let
$\tt X_t, \tt Y_t$ be two solutions of \eqref{E1} with $\tt X_0= \tt Y_0$ and life times $\xi,\eta$ respectively. Let
$$\xi_n:=\inf\{t\in [0,T]: |\tt X_t|\ge n\},\ \ \eta_n:=\inf\{t\in [0,T]: |\tt Y_t|\ge n\},\ \ n\ge 1.$$
Let $T_n=\xi_n\land\eta_n.$ It remains to prove $\mathbb{P}$-a.s.
\beq\label{UN}   |\tt X_{t\land T_n}-\tt Y_{t\land T_n}|=0,\ \ n\ge 1, t\in [0,T].\end{equation}
Let $h_n\in C_0^\infty(\R^N)$ such that $h_n|_{B(0,n)}=1$. Then, up to time $ T\land T_n$, $\tt X_t$ and $ \tt Y_t$ solve the SDE \eqref{E1} for $h_n\b$ replacing $\b$.

By \eqref{QN2}, we take large enough $\ll>0$  such that
 $$\sup_{t\in[0,T]} \|h_n\nn \Xi_\ll (h_n\b)_t\|_\infty\le \ff 1 2.$$
 Simply denote $\u= \Xi_\ll(h_n\b)$ and $\theta_s(x)= x+ \u_s(x)$. Then
 \beq\label{QN3} \ff 1 2 |\theta_t(x)-\theta_t(y)|\le |x-y|\le 2 |\theta_t(x)-\theta_t(y)|,\ \ t\in [0,T], x,y\in B(0,n).\end{equation}

By Lemma \ref{L2.3} and It\^{o}'s formula,   we have
\beq\label{QN4}\beg{split}
&|\theta_{t\land T_n}(\tt X_{t\land T_n})-\theta_{t\land T_n}(\tt Y_{t\land T_n})|^2\\
&=2  \int_0^{t\land T_n} \big\< Z_s(\tt X_s)- Z_s(\tt Y_s) + \ll ( \u_s(\tt X_s)-\u_s(\tt Y_s)), \theta_s(\tt X_s)-\theta_s(\tt Y_s)\big\> \d s\\
&\quad +\int_{0}^{t\land T_n} \left\|[\nabla_{\sigma_s(\tt X_s)} \theta_s(\tt X_s)-\nabla_{\sigma_s(\tt Y_s)} \theta_s(\tt Y_s)]\right\|_{\mathrm{HS}}^2\d s\\
&\quad +2 \int_0^{t\land T_n}\big\<\nn_{\si_s(\tt X_s)(\theta_s(\tt X_s)-\theta_s(\tt Y_s))}\theta_s(\tt X_s) -  \nn_{\si_s(\tt Y_s)(\theta_s(\tt X_s)-\theta_s(\tt Y_s))}\theta_s(\tt Y_s), \d B_s\big\>\\
&= \int_0^t \bb_n(s)|\theta_s(\tt X_s)-\theta_s(\tt Y_s)|^2\d s + \int_0^t \<\alpha_n(s)|\theta_s(\tt X_s)-\theta_s(\tt Y_s)|^2,\d B_s\>,\ \ t\in [0,T],
\end{split}\end{equation} where
\beg{align*} \bb_n(s)&:= \ff{1_{\{s<T_n\}} 1_{\{\tt X_s\ne \tt Y_s\}}}{|\theta_s(\tt X_s)-\theta_s(\tt Y_s)|^2} \Big(\left\|\nabla_{\sigma_s(\tt X_s)} \theta_s(\tt X_s)-\nabla_{\sigma_s(\tt Y_s)} \theta_s(\tt Y_s)\right\|_{\mathrm{HS}}^2\\
&\qquad +2 \big\< Z_s(\tt X_s)- Z_s(\tt Y_s) + \ll ( \u_s(\tt X_s)-\u_s(\tt Y_s)),
   \theta_s(\tt X_s)-\theta_s(\tt Y_s)\big\>  \Big),\\
 \aa_n(s)&:= \ff{21_{\{s<T_n\}} 1_{\{\tt X_s\ne \tt Y_s\}}}{|\theta_s(\tt X_s)-\theta_s(\tt Y_s)|^2}\Big(\nn_{\si_s(\tt X_s)(\theta_s(\tt X_s)-\theta_s(\tt Y_s))}\theta_s(\tt X_s) -  \nn_{\si_s(\tt Y_s)(\theta_s(\tt X_s)-\theta_s(\tt Y_s))}\theta_s(\tt Y_s)\Big).\end{align*}
 Since $h_n|_{B(0,n)}=1$, $\bb_n$ and $\aa_n$ do not change if $Z, \nabla_\sigma\theta,$ and $\u$ are replaced by $h_n Z, h_n\nabla_\sigma\theta$ and $h_n \u$ respectively.   So, letting
$$\Phi_s= \|\nn ( h_n Z)_s\| + \|\nn (h_n \u)_s\| +\|\nn (h_n\nn_{\si_s}\theta_s)\|^2,$$
by Lemma \ref{Hardy} we may find a constant $C_1>0$ such that
\beq\label{QN5}|\aa_n(s)|^2+ |\bb_n(s)|\le    C_11_{\{s<T_n\}}  (\scr M\Phi_s(\tt X_s)+ \scr M \Phi_s(\tt Y_s)),\ \ s\in [0,T].\end{equation}
Applying  Theorem \ref{T2.1}(1) for $h_n\b$ replacing $\b$ and using \eqref{QN5},  we obtain
$$\E  \bigg(\int_s^T (|\aa_n(s)|^2+|\bb_n(s)|)\d s\bigg|\F_s\bigg) \le  \kk \|\scr M\Phi\|_{L^{q}_{p}},\ \ s\in [0,T] $$
for some constant $\kk>0$. Since  Lemma \ref{Hardy}  and  our conditions in Theorem \ref{T2.0}(3) imply
$$\|\scr M\Phi\|_{L^{q}_{p}}\le \kk' \|\Phi\|_{L^{q}_{p}}<\infty$$   for some constant $\kk'>0$, using the Khasminskii estimate as in \eqref{QE}
we conclude that
$$\E\exp\bigg[c \int_0^T (|\aa_n(s)|^2+|\bb_n(s)|)\d s\bigg]<\infty,\ \ c>0.$$
So, by Dol\`{e}ans-Dade's exponential formula, \eqref{QN4} implies
$$|\theta_{t\land T_n}(\tt X_{t\land T_n})-\theta_{t\land T_n}(\tt Y_{t\land T_n})|^2= |\theta_0(\tt X_0)-\theta_0(\tt Y_0)|^2 \e^{2\int_0^t \<\aa_n(s),\d B_s\>+\int_0^t\big(\bb_n(s)-2 |\aa_n(s)|^2\big)\d s},\ \ t\in [0,T].$$
Since $\tt X_0=\tt Y_0$, we have proved \eqref{UN}.

\end{proof}

\section{Singular SDEs on generalized Heisenberg groups}

\subsection{Framework and main result}

Consider the following vector fields on $\R^{m+d}$, where $m\ge 2, d\ge 1$:
\beq\label{3.1} U_i(x,y)= \sum_{k=1}^m\theta_{ki} \pp_{x_k} + \sum_{l=1}^d (A_l x)_i\pp_{y_l},\  \ 1\le i\le m,\end{equation}
where $(x,y)=(x_1,\cdots, x_m, y_1,\cdots, y_d)\in \R^{m+d}$, $\Theta:=(\theta_{ij})$ and $A_l (1\leq l\le d)$ are $m\times m$-matrices satisfying the following assumption:
\beg{enumerate} \item[$(H)$] $\alpha$ is invertible, $G_l:= A_l \alpha- \alpha^*A_l^*\ne 0 (1\le l\le d)$, and there exists $\vv\in [0,1)$ such that
$$\vv \sum_{l=1}^d a_l^2 |G_l u|^2\ge \sum_{1\le l\ne k\le d} |a_la_k \<G_lu, G_ku\>|,\ \ a\in \R^d, u\in \R^m.$$\end{enumerate}

As showing in the beginning of \cite[\S 1]{W9}, this assumption implies
\beq\label{3.2} \sum_{i,j=1}^m \bigg|\sum_{l=1}^d (G_l)_{ij} a_l\bigg|^2 \ge (1-\vv) \Big(\inf_{1\le l\le d} \|G_l\|_{HS}^2\Big) |a|^2,\ \ a\in \R^d.\end{equation} Consequently, $\{U_i, [U_i,U_j]\}_{1\le i,j\le m}$ spans the tangent space of $\R^{m+d}.$ Since $\text{div} U_i=0$, the operator
$$\L:= \ff 1 2\sum_{i=1}^m U_i^2$$ is subelliptic and symmetric in $L^2(\R^{m+d}),$ and the associated  diffusion process solves the SDE for
$(X_t,Y_t)\in \R^{m+d}$:
\beq\label{E3.1} \d (X_t,Y_t) =\sum_{i=1}^m U_i(X_t)\circ \d B_t^i = Z\d t + \si(X_t) \d B_t,\end{equation}
where $B_t:=(B_t^i)_{1\le i\le m} $ is the $m$-dimensional Brownian motion, and
$$\si(x):= (\Theta, A_1 x,\cdots, A_d x),\ \ Z:=\sum_{i=1}^m \nn_{U_i}U_i= \sum_{l=1}^d{\rm tr}(\Theta A_l)\pp_{y_l}.$$
We now consider the following SDE with a singular drift $\b: [0,T]\times \R^{m+d} \to \R^m$:
\beq\label{E3.2}   \d (\tt X_t,\tt Y_t)   = \big\{\si(\tt X_t) \b_t(\tt X_t,\tt Y_t) +Z\big\}\d t + \si(\tt X_t) \d B_t.\end{equation}

\paragraph{Remark 3.1.} Take $d=m-1, \Theta=I_{m\times m}$ and for  some constants $a_l\ne\bb_l$,
$$(A_l)_{ij}= \beg{cases} a_l, &\text{if}\ i=1, j= l+1,\\
\bb_l, &\text{if}\ i=l+1, j= 1,\\
0, &\text{otherwise}.\end{cases}$$ Then $G_l^*G_k=0$ for $l\ne k$, so that $(H)$ holds with $\vv=0.$ In particular, for $a_l=-\bb_l=\ff 1 2$, $\L$ is the Kohn-Laplacian operator on the $(2m-1)$-dimensional Heisenberg group. In general, $\R^{m+d}$ is a group under the action
\beq\label{GR} (x,y)\bullet (x',y'):= (x+x', y+y'+ \<(\Theta^\ast)^{-1} A_\cdot x, x'\>),\ \ (x,y), (x',y')\in \R^{m+d},\end{equation} and $U_i, 1\le i\le m$ are left-invariant vector fields. So, we call \eqref{E3.2} a singular SDE on the generalized Heisenberg group.

\

For two nonnegative functions $F_1,F_2$,  we write $F_1\preceq F_2$ if there exists a constant $C>0$ such that
$F_1\le CF_2$, and write $F_1\asymp F_2$ if $F_1\preceq F_2$ and $F_2\preceq F_1.$

Let $\Delta_y=\sum_{l=1}^d \p_{y_l}^2$. Then $(\Delta_y, W^{2,2}(\R^d))$ is a negative definite operator in $L^2(\R^d)$. For any $\aa>0$ and $\ll\ge 0$, we consider the operator    $(\ll-\DD_y)^\aa$   defined on domain $\D((-\DD_y)^\aa):=W^{2\aa,2}(\R^d)$. This operator  extends naturally to a measurable function $f$ on the produce space $\R^{m+d}$ such that  $f(x,\cdot)\in \D((-\DD_y)^\aa)$ for   $x\in \R^m$:
$$(\ll-\DD_y)^\aa f(x,y):= (\ll-\DD_y)^\aa f(x,\cdot)(y).$$ For any $\bb>0, p\ge 1$, let $\H_y^{\aa,p}$ be the space of measurable functions on $\R^{m+d}$ such that
$$\|f\|_{\H_y^{\bb, p}}:= \|(1-\DD_y)^{\ff \bb 2}f\|_p \asymp \|f\|_p + \|(-\DD_y)^{\ff \bb 2}f\|_p<\infty.$$ Recall that for $\bb\in (0,2)$, we have
\begin{equation}\begin{split}\label{D_y}
-(-\Delta_y) ^{\frac{\beta}{2}}f(z):=\int_{\mathbb{R}^{d}}(f(z+(0,y'))-f(z))|y'|^{-(m+\beta)}\d y',\ \ z\in \R^{m+d}.
\end{split}\end{equation}

For any $\bb>0, p,q\ge 1$,  let $\H_y^{\bb,p,q}$ be the completion of $C_0^\infty([0,T]\times\R^{m+d})$ with respect to the norm
$$\|f\|_{\H_y^{\bb,p,q}}:=\|(1-\DD_y)^{\ff\bb 2}f\|_{L_p^q}\asymp\|f\|_{L^q_p}+ \|(-\Delta_y)^{\ff\bb 2} f\|_{L^q_p}.$$

Applying Theorem \ref{T2.0} to the present model, we will prove the following result.

\beg{thm}\label{T3.1} Assume $(H)$ and let $p,q\ge 1$ satisfy
\beq\label{H*}\frac{2}{q}+\frac{m+2d}{p}<1.\end{equation}
\beg{enumerate} \item[$(1)$] If $|\b|\in L^{q}_p([0,\infty]\times\R^N)$, then for any initial value $x\in \R^{m+d}$, the SDE $\eqref{E3.2}$ has a weak solution $(X_t)_{t\in [0,T]}$ starting at $x$  with $\E\e^{\ll \int_0^T |\b_t(X_t)|^2\d t}<\infty$ for all $\ll>0$.
\item[$(2)$] If $ (h\b) \in \H_y^{\ff 1 2, p,q}$ holds for any  $h\in C_0^\infty(\R^{m+d})$,
then for any initial value $x\in \R^{m+d}$, the SDE $\eqref{E3.2}$ has a unique strong solution up to life time.\end{enumerate}  \end{thm}

\subsection{Proof of Theorem \ref{T3.1}}

To apply Theorem \ref{T2.0}, we first collect some known assertions about $\L$ and the associated Markov semigroup $P_t$.
Let $\|\cdot\|_{p\to q}$ denote the operator norm from $L^p(\R^{m+d})$ to $L^q(\R^{m+d}),$ and let $\|\cdot\|_p= \|\cdot\|_{p\to p}.$
For any $\aa>0,p\ge 1$, let $\H_\si^{\aa,p}$ be the completion of $C_0^\infty(\R^{m+d})$ with respect to the norm
$$\|f\|_{\H_\si^{\aa,p}}:= \|(1-\scr L)^{\ff \aa 2 }f\|_p \asymp \|f\|_p+ \|(-\L)^{\ff\aa 2} f\|_p.$$

It is classical that
\begin{equation}\begin{split}\label{half}
\|(-\Delta_y)^{\frac{1}{2}} f\|_{p}\asymp  \|\nabla_y f\|_{p},
\end{split}\end{equation}
and for any $\bb>0$,
\begin{equation}\begin{split}\label{eqnorm}
&\|f\|_{\mathbb{H}_y^{\beta,p}}\asymp\|f\|_{p}+\|(-\Delta_y)^{\frac{\beta-[\beta]}{2}}\nabla_y^{[\beta]} f\|_{p},
\end{split}\end{equation}
where $[\beta]:=\sup\{k\in \mathbb Z_+: k\le\bb\}$  is the integer part of $\bb$.

Moreover, by the interpolation inequality, for any $0\leq\alpha<\beta<\infty$ we have
\begin{equation}\begin{split}\label{ip0}
&\|f\|_{\mathbb{H}_y^{\alpha,p}}\preceq  \|f\|^{\frac{\beta-\alpha}{\beta}}_{p}\|f\|^{\frac{\alpha}{\beta}}_{\mathbb{H}_y^{\beta,p}}.
\end{split}\end{equation}

\beg{lem}\label{L3.1} Assume $(H)$.
\beg{enumerate} \item[$(1)$] There exists a constant $C>0$ such that
\beq\label{G01} \|P_t\|_{L^1\to L^\infty}\le C t^{-\ff{m+2 d}2},\ \ t>0.\end{equation} Moreover, for any $p>1$ there exists a constant $c_p>0$ such that
\beq\label{G1}|\nn_{\sigma} P_t f|\le \ff{c_p}{\ss t} (P_t |f|^p)^{\ff 1 p},\ \ f\in \B_b(\R^{m+d}), t>0.\end{equation}
\item[$(2)$] For any $r\ge 0, p\in (1,\infty)$,
$$\|(1-\L)^{r+\ff 1 2}f\|_p \asymp \|(1-\L)^{r}f\|_p+ \|(1-\L)^{r}\nn_\si f\|_p,\ \ f\in \H_\si^{1+2r,p},$$
and
$$\|(-\L)^{r+\ff 1 2}f\|_p \asymp \|(-\L)^{r}\nn_\si f\|_p,\ \ f\in \H_\si^{1+2r,p}.$$
\item[$(3)$] For any $r>0$ and $p\in (1,\infty),$
$$\|(1-\L)^{r}f\|_p \asymp \|f\|_p + \|(-\L)^r f\|_p,\ \ f\in\H_\si^{2r,p}.$$
\item[$(4)$] For any $r\in (0,1)$ and $p>\ff{m+2d}{2r}$,
$$\|f\|_\infty\preceq \|(1-\L)^{r}f\|_p,\ \ f\in \H_\si^{2r,p}.$$
\item[$(5)$] For any $p\in (1,\infty),\aa_1,\aa_2\ge 0, \theta\in (0,1)$,  and $f\in \H_\si^{2\aa_2,p}\cap \H_y^{2\aa_1,p},$
\beg{align*} &\|(1-\Delta_y)^{\theta \aa_1} (1-\L)^{(1-\theta)\aa_2} f\|_p\preceq \|(1-\Delta_y)^{\aa_1} f\|^\theta_p \|(1-\L)^{\aa_2}f\|^{1-\theta}_p, \\
&\|(-\Delta_y)^{\theta \aa_1} (-\L)^{(1-\theta)\aa_2} f\|_p\preceq \|(-\Delta_y)^{\aa_1} f\|^\theta_p \|(-\L)^{\aa_2}f\|^{1-\theta}_p.\end{align*}
\end{enumerate}
\end{lem}

\beg{proof} The inequalities in  (1) follow  from Lemma 2.4 and Corollary 1.2 in \cite{W9} respectively. Assertion (2) is due to \cite[Theorem 4.10]{FGB}. Since $P_t$ is contractive in $L^p(\R^{m+d})$ and
\beq\label{SA} (1-\L)^{-\aa}= c \int_0^\infty \e^{-t} t^{\aa-1} P_t \d t\end{equation} for some constant $c>0$, $(1-\L)^{-\aa}$ is bounded in $L^p(\R^{m+d})$ for all  $p\ge 1$. Combining this with the closed graph theorem  that
$$\|f\|_p+ \|(-\L)^\aa f\|_p\asymp \|f\|_p+\|(1-\L)^\alpha f\|_p,$$ we prove assertion (3).  By the first inequality in assertion (1) and using \eqref{SA}, we have
$$\|(1-\L)^{-\ff r 2} f\|_\infty \preceq \|f\|_p \int_0^\infty \e^{-t}\ff{\|P_t\|_{p\to\infty}}{t^{1-r}}  \d t  \le C \|f\|_p$$ for some constant $C>0$. So, assertion (4) holds.
Finally, let
$$\scr A= (1-\Delta_y)^{\aa_1} (1-\L)^{-\aa_2}.$$ By the interpolation theorem (see \cite[Theorem 6.10]{PA}), we have
$$\|\scr A^\theta g\|_p\preceq \|g\|^{1-\theta}_p \|\scr Ag\|^{\theta}_p.$$ Applying this inequality to $g= (1-\L)^{\aa_2}f$, we obtain
$$\|(1-\Delta_y)^{\theta \aa_1} (1-\L)^{(1-\theta)\aa_2} f\|_p= \|\scr A^\theta g\|_p \preceq \|(1-\L)^{\aa_2}\|_p^{1-\theta} \|(1-\Delta_y)^{\aa_1}\|_p^\theta.$$ \end{proof}

\beg{proof}[Proof of Theorem \ref{T3.1}]  We first estimate $\scr K_1$ and $\scr K_2$.
Let $P_{s,t}=P_{t-s}$. By \eqref{G01} and using the interpolation theorem, we have
 \beq\label{QP01} \|P_{s,t} f\|_\infty\preceq (t-s)^{-\ff{m+2d}{2p}} \|f\|_p, t>s\ge 0, p\ge 1.\end{equation}So,
 \beq\label{KK1}\scr K_1\supset \Big\{(p,q)\in (1,\infty]^2:\ \ff 1 q +\ff{m+2d}{2p}< 1\Big\}. \end{equation} Combining \eqref{QP01}   with \eqref{G1}, we see that  for any $\vv\in (0,p-1),$
 \beq\label{QP02}   \|\nn_\si P_{s,t}f\|_\infty \preceq (t-s)^{-\ff 1 2} \|P_{s,t} |f|^{1+\vv}\|_\infty^{\ff 1 {1+\vv}} \preceq (t-s)^{-\ff 1 2 - \ff{(m+2d)(1+\vv)}{2p}} \|f\|_p, \ \ t>s\geq 0.\end{equation} So,
\beq\label{KK2}\scr K_2\supset \Big\{(p,q)\in (1,\infty]^2:\ \ff 2 q +\ff{m+2d}{p}< 1\Big\}. \end{equation}
 Therefore,  the first assertion follows from Theorem \ref{T2.0}(2).

 Next, we verify   $(A_1)$, $(A_2)$ and the assumption in Theorem \ref{T2.0}(3).
  Since $\Theta$ is invertible, there exists a constant $\ll>0$ such that
$$|\si v|\ge |\Theta v|\ge \ll |v|,\ \ v\in \R^m.$$ So, $(A_1)$ holds.
Next, since $U_i$ are smooth vector fields with  constant or linear coefficients,   $\pp_t P_t f= \L P_t f$  for $f\in C_0^\infty(\R^N)$ and
  $$\|\nn P_t f\|_\infty\le C\|\nn f\|_\infty,\ \ t\in [0,T], f\in C_b^1(\R^N)$$ for some constant $C>0$.
So, $(A_2)$ and the assumption in Theorem \ref{T2.0}(3) \textcolor{red}{(i)} hold.
So, for  $(p,q)$ satisfy \eqref{H*}, by \eqref{KK1} and \eqref{KK2}  we have $(p,q)\in \scr K_1\cap\scr K_2$. According to Theorem \ref{T2.0}(3), it remains to prove that for $h\in C_0^\infty(\R^{m+d})$,
\beq\label{D1}\lim_{\ll\to\infty} \|\nn \{h\Xi_\ll (h\b)\}\|_\infty=0,\end{equation}
\beq\label{D1'}\limsup_{\ll\to\infty} \|\nn \{h\nn_\si \Xi_\ll (h\b)\}\|_{L_{p}^{q}}<\infty.\end{equation}
 We leave the proofs to  the following subsection.
\end{proof}

\subsection{Proofs of \eqref{D1} and \eqref{D1'}}

 We first  investigate the regularity of the solution to the following PDE:
\begin{align}\label{PDEf}
\pp_t u_t= (\ll-\L) u_t -f_t,\ \ u_T=0.
\end{align}
For this, we need some preparations.

The following interpolation theorem comes from \cite{FGB,S}.
\begin{lem}\label{Int} Let $p\in(1,\infty)$, $0<\alpha<\beta$ and $ f\in \H_\si^{2\aa,p}\cap \H_\sigma^{2\beta,p}$. For any $\theta\in(0,1)$, let $\gamma=\theta\alpha+(1-\theta)\beta$. Then $f\in\H_\si^{2\gamma,p}$ and
\begin{align*}
&\|(-\L)^\gamma f\|_p\leq C\|(-\L)^\alpha f\|^{\theta}_p\|(-\L)^\beta f\|^{1-\theta}_p,\\
&\|(1-\L)^\gamma f\|_p\leq C\|(1-\L)^\alpha f\|^{\theta}_p\|(1-\L)^\beta f\|^{1-\theta}_p,
\end{align*}
where $C$ only depends on $\alpha,\beta,\gamma$.
\end{lem}

Next,  let $P_t$ be the diffusion semigroup associated with the SDE \eqref{E3.1}. We estimate derivatives of $P_t$ by following the line of \cite{W9}.

\begin{lem}\label{GE0} Let $p>1$, $t>0$. Then the following assertions hold.
\begin{enumerate}
\item[$(1)$] There exists a constant $c_p>0$ such that for any $f\in\B_b(\mathbb{R}^{m+d})$,
\begin{equation}\begin{split}\label{G2}
|\nabla_y P_tf|\leq \frac{c_p}{t}(P_t |f|^p)^{\frac{1}{p}},
\end{split}\end{equation}
and
\begin{equation}\begin{split}\label{G3}
|\nabla_\sigma\nabla_\sigma P_tf|\leq\frac{c_p}{t}(P_t |f|^p)^{\frac{1}{p}}.
\end{split}\end{equation}
\item[$(2)$] For any $\alpha\in(0,1)$,
there exists a constant $C=C(p,\alpha)$
such that for all $f\in L^{p}(\mathbb{R}^{m+d})$,
\begin{equation}\begin{split}\label{GE1}
&\|\nabla_\sigma P_t f\|_{\mathbb{H}_\sigma^{\alpha,p}}+\|(-\Delta_y)^{\frac{1}{4}}P_t f\|_{\mathbb{H}_\sigma^{\alpha,p}}\leq C t^{-\frac{\alpha}{2}-\frac{1}{2}}\|f\|_{p}.
\end{split}\end{equation}
\end{enumerate}
\end{lem}
\begin{proof}   \eqref{G2} follows from   \cite[Theorem 1.1]{W9} for $u=0$. Moreover, combining \eqref{G2} with \eqref{half} and \eqref{ip0}, we obtain
\beq\label{W*B} \|(-\DD_y)^{\ff\bb 2}P_t f\|_p\preceq \ff 1 {t^\bb} (P_t|f|^p)^{\ff 1 p},\ \ \bb\in (0,1), p>1, t>0.\end{equation}
Then we can claim that   that it suffices to prove \eqref{G3} holds. Indeed,  by \eqref{G3},   \eqref{G1} and  Lemma \ref{Int} we obtain
$$\|\nabla_\sigma P_t f\|_{\mathbb{H}_\sigma^{\alpha,p}}\preceq t^{-\frac{\alpha}{2}-\frac{1}{2}}\|f\|_{p}$$ for $\aa\in (0,1)$. On the other hand, by  Lemma \ref{L3.1}(5), \eqref{G3} and \eqref{W*B}, we have
\begin{align*}
\|(-\Delta_y)^\frac{1}{4}(-\L)^{\frac{\alpha}{2}}P_tf\|_p&=\|(-\Delta_y)^{\frac{1}{4(1-\frac{\alpha}{2})}(1-\frac{\alpha}{2})}(-\L)^\frac{\alpha}{2}P_tf\|_p\\
&\preceq \|(-\Delta_y)^{\frac{1}{4(1-\frac{\alpha}{2})}}P_tf\|^{1-\frac{\alpha}{2}}_p\|(-\L)P_tf\|^\frac{\alpha}{2}_p\\
&\preceq t^{-\frac{1}{2}}t^{-\frac{\alpha}{2}}\|f\|_p.
\end{align*} Therefore, \eqref{GE1} holds.

We now prove \eqref{G3} by  using the   derivative formula in \cite[Theorem 1.1]{W9}. Let $\QQ_t=(q_{kl}(t))_{1\leq k,l\leq d}$ with
$$q_{lk}(t):=\int_0^t\left\langle G_l^\ast G_k\left(B_s-\frac{1}{t}\int_0^t B_s\d s\right),\left(B_s-\frac{1}{t}\int_0^t B_s\d s\right)\right\rangle\d s.$$ Then $\QQ_t$ is invertible for $t>0.$ Next,
  for $x,w\in\R^m$ and $v\in \R^d$, let
\beq\label{alpha}
 (\tilde{\alpha}_{t,w,v,x})_l:=v_l-\langle \Theta^{-1}w,A_lx\rangle-\frac{1}{t}\int_0^t\langle G_l^\ast \Theta^{-1}w,B_s\rangle\d s,\ \ 1\le l\le d.\end{equation} Then for the functional $(x,y)\mapsto \tt\aa_{t,w,v,x}$ we have
 \beq\label{*WD} \nabla_{(w',v')}(\tilde{\alpha}_{t,w,v,x})_l=-\langle \Theta^{-1}w,A_lw'\rangle, \ \ (w',v')\in \R^{m+d}, 1\leq l\leq d.
\end{equation}
Next,  the solution of \eqref{E3.1} starting at $(x,y)$ is given by
\begin{align*}
 X_t=x+\alpha B_t,\ \ (Y_t)_l=y_l+\langle A_lx, B_t\rangle+\int_0^t\langle A_l\alpha B_s,\d B_s\rangle,\ \ 1\leq l\leq d.
\end{align*}
Then
\begin{align}\label{GS0}
&\nabla_{(w',v')} (X_t,(Y_t)_l)=(w',(v')_l'+\langle A_lw',B_t\rangle),\ \ 1\leq l\leq d.
\end{align}
According to   \cite[Theorem 1.1(3)]{W9}, we have the Bismut derivative formula
\begin{align}\label{Bismut}
\nabla_{w,v}P_t f=\mathbb{E}[f(X_t,Y_t)M_t],
\end{align}
 where  by the formulation of $\tt h'$ given in \cite[Theorem 1.1]{W9},
\beq\label{*WD2}\beg{split}
M_t &:=D^*\tt h =  \frac{1}{t}\langle \Theta^{-1}w,B_t\rangle+\sum_{k=1}^d (\QQ_t^{-1}\tilde{\alpha}_{t,w,v,x})_k\int_0^t \langle G_kB_s,\d B_s\rangle\\
&-\sum_{k=1}^d D_{\beta_k} (\QQ_t^{-1}\tilde{\alpha}_{t,w,v,x})_k-\sum_{k=1}^d\frac{(\QQ_t^{-1}\tilde{\alpha}_{t,w,v,x})_k}{t}\left\langle\int_0^t G_kB_s\d s,B_t\right\rangle\\
&+\sum_{k=1}^d\sum_{i=1}^m\frac{D_{h_i}(\QQ_t^{-1}\tilde{\alpha}_{t,w,v,x})_k}{t}\int_0^t (G_kB_s)_i\d s+\sum_{k=1}^d\sum_{i=1}^m\frac{t}{2}(\QQ_t^{-1}\tilde{\alpha}_{t,w,v,x})_k(G_k)_{ii},
\end{split}\end{equation}
for $h_i(s):=s e_i, \beta_k(s):=\int_0^s G_kB_r\d r, s\in[0,t]$, and $\{e_i\}_{i=1,\cdots,m}$ being the  orthonormal basis of $\mathbb{R}^m$.
According to step (1) in the proof of  \cite[Theorem 1.1]{W9}, for any $p>1$, we have
\begin{align}\label{M^p}
\{\mathbb{E}|M_t|^p\}^{1/p}\preceq\frac{(|v|+|w|(|x|+\sqrt{t}))}{t}.
\end{align}
Moreover, by \eqref{*WD2} we have
\begin{align*}
\nabla_{(w',v')}M_t&=\sum_{k=1}^d(\QQ_t^{-1}\nabla_{(w',v')}\tilde{\alpha}_{t,w,v,x})_k\int_0^t \langle G_kB_s,\d B_s\rangle\\
&-\sum_{k=1}^d D_{\beta_k} (\QQ_t^{-1}\nabla_{(w',v')}\tilde{\alpha}_{t,w,v,x})_k-\sum_{k=1}^d\frac{(\QQ_t^{-1}\nabla_{(w',v')}\tilde{\alpha}_{t,w,v,x})_k}{t}\left\langle\int_0^t G_kB_s\d s,B_t\right\rangle\\
&+\sum_{i=1}^m\sum_{k=1}^d\frac{D_{h_i}(\QQ_t^{-1}\nabla_{(w',v')}\tilde{\alpha}_{t,w,v,x})_k}{t}\int_0^t (G_kB_s)_i\d s\\
&+\sum_{k=1}^d\sum_{i=1}^m\frac{t}{2}(\QQ_t^{-1}\nabla_{(w',v')}\tilde{\alpha}_{t,w,v,x})_k(G_k)_{ii}.
\end{align*}
Combining this with \eqref{*WD} we prove
\begin{align}\label{GM^p}
\{\mathbb{E}|\nabla_{(w',v')}M_t|^p\}^{1/p}\preceq\frac{|w||w'|}{t},\ \ (w',v')\in \R^{m+d}.
\end{align}
By the Markov property and \eqref{Bismut}, we derive
\begin{align*}
\nabla_{(w,v)}P_tf=\nabla_{(w,v)}P_{\frac{t}{2}}(P_{\frac{t}{2}}f)=\mathbb{E}[(P_{\frac{t}{2}}f)(X_\frac{t}{2},Y_\frac{t}{2})M_\frac{t}{2}],
\end{align*}
and by the chain rule,
\beq\label{*WD3}\beg{split}
&\nabla_{(w',v')}\nabla_{(w,v)}P_tf=\nabla_{(w',v')}\mathbb{E}[(P_{\frac{t}{2}}f)(X_\frac{t}{2},Y_\frac{t}{2})M_\frac{t}{2}]\\
&=\mathbb{E}\left[\left(\nabla_{\nabla_{(w',v')}(X_\frac{t}{2},Y_\frac{t}{2})}P_{\frac{t}{2}}f\right)(X_\frac{t}{2},Y_\frac{t}{2}) M_\frac{t}{2}\right]
  +\mathbb{E}\left[(P_{\frac{t}{2}}f)(X_\frac{t}{2},Y_\frac{t}{2})\nabla_{(w',v')}M_\frac{t}{2}\right].
\end{split}\end{equation}
By \eqref{GS0}, \eqref{M^p} and using  H\"{o}lder's inequality, we obtain
\begin{align*}
&\mathbb{E}\left|\left(\nabla_{\nabla_{(w',v')}(X_\frac{t}{2},Y_\frac{t}{2})}P_{\frac{t}{2}}f\right)(X_\frac{t}{2},Y_\frac{t}{2})M_\frac{t}{2}\right|\\
&\preceq(P_t|f|^p)^{1/p}\frac{(|v'|w'(|x|+\sqrt{t}))(|v|+w(|x|+\sqrt{t}))}{t^2},
\end{align*} while by \eqref{GM^p} and
  H\"{o}lder's inequality,
\begin{align*}
\mathbb{E}\left|(P_{\frac{t}{2}}f)(X_\frac{t}{2},Y_\frac{t}{2})\nabla_{(w',v')}M_\frac{t}{2}\right|\preceq(P_t|f|^p)^{1/p}\frac{|u||u'|}{t}.
\end{align*}
Therefore, it follows from \eqref{*WD} that
\beq\label{g2p}\beg{split}
& |\nabla_{(w',v')}\nabla_{(w,v)}P_tf|(x,y)\\
 &\preceq(P_t|f|^p)^{1/p}\left(\frac{(|w'|+v'(|x|+\sqrt{t}))(|v|+w(|x|+\sqrt{t}))}{t^2}+\frac{|w|\cdot|v'|}{t}\right).
\end{split}\end{equation}
Finally, by the definition of $U_i$,  we have
\begin{align*}
U_iU_j&=\left(\sum_{k=1}^m\theta_{ki} \pp_{x_k} + \sum_{l=1}^d (A_l x)_i\pp_{y_l}\right)\left(\sum_{k=1}^m\theta_{kj} \pp_{x_k} + \sum_{l=1}^d (A_l x)_j\pp_{y_l}\right)\\
&=\sum_{k=1}^m\sum_{l=1}^m\theta_{ki}\theta_{lj} \pp_{x_k} \pp_{x_l}+\sum_{k=1}^m\sum_{l=1}^d\theta_{ki} (A_l)_{jk}\pp_{y_l}+\sum_{k=1}^m\sum_{l=1}^d\theta_{ki}(A_l x)_j\pp_{x_k}\pp_{y_l}\\
&+\sum_{l=1}^d \sum_{k=1}^m(A_l x)_i\theta_{kj}\pp_{y_l} \pp_{x_k}+\sum_{l=1}^d\sum_{k=1}^d (A_l x)_i(A_k x)_j\pp_{y_l}\pp_{y_k}.
\end{align*}
Combining this with \eqref{G2} and \eqref{g2p} with $(x,y)=(0,0)$, we arrive at
\begin{align*}
&|U_i U_jP_tf(0,0)|\preceq\frac{1}{t}(P_t|f|^p(0,0))^{1/p},\ \ 1\leq i,j\leq m, p>1, t>0.
\end{align*}
As explained in the proof of \cite[Proof of Corollary 1.2]{W9}, by the left-invariant property of $U_i$ and $\partial_{y_l}$ under the group action in \eqref{GR}, this is equivalent to  \eqref{G3}.
\end{proof}

The next lemma due to \cite[Theorem 5.15]{FGB} generalizes the classical Sobolev embedding theorem.
\begin{lem} \label{SET} Suppose $p\in(1,\infty)$ and $\alpha>\frac{m+2d}{p}$, then
\begin{equation}\label{fgb}
\|f\|_{\Gamma_\gamma}\leq C(p,m+2d,\alpha)\|f\|_{\mathbb{H}_\sigma^{\alpha,p}}, \ \ \gamma\in[0,\alpha-(m+2d)/p],
\end{equation}
where
\begin{equation*}
\|f\|_{\Gamma_\gamma}:=\|f\|_\infty+|f|_\gamma, \ \ |f|_\gamma:=\sup_{x\in\mathbb{R}^{d+m},y\neq0}\frac{|f(x\bullet y)-f(x)|}{|y|^\gamma}.
\end{equation*}
\end{lem}

Finally, we introduce the following lemma.

\begin{lem} \label{b-gra} Let $p>m+2d$. For any $\beta\in(0,1]$ and $\alpha\in(\frac{m+2d}{p},1]$, there exists a constant $C =C(\alpha,\beta,m+2d,p)>0$ such that for  $\R^m$-valued function $\b\in \mathbb{H}_y^{\beta,p}$ and real function $u\in L^p(\mathbb{R}^{m+d})$ with $(-\L)^{\frac{1}{2}+\frac{\alpha}{2}}u\in\mathbb{H}_y^{\beta,p}$,
\begin{equation*}\begin{split}
\|  \nabla_{\sigma \b}u \|_{\mathbb{H}_y^{\beta,p}}\preceq\|\b\|_{\mathbb{H}_y^{\beta,p}}\left(\|(-\Delta_y)^{\frac{\beta}{2}}\nabla_\sigma u\|_{\mathbb{H}_\sigma^{\alpha,p}}+\left\|\nabla_\sigma u\right\|_{\mathbb{H}_\sigma^{\alpha,p}}\right).
\end{split}\end{equation*}
\end{lem}
\begin{proof} By the definition of $\|\cdot\|_{\mathbb{H}_y^{\beta,p}}$ and noting that $\nn_{\si \b}u=\<\nn_\si u, \b\>$, we have
\begin{equation}\begin{split}\label{gbsbet}
\|\nabla_{\sigma\b} u \|_{\mathbb{H}_y^{\beta,p}}&\preceq \|\langle \nabla_\sigma u,\b\rangle\|_{p}+\|(-\Delta_y)^{\frac{\beta}{2}}\langle \nabla_\sigma u,\b\rangle\|_{p}\\
&\preceq \|\nabla_\sigma u\|_{\infty}\|\b\|_{p}+\|(-\Delta_y)^{\frac{\beta}{2}}\langle \nabla_\sigma u,\b\rangle\|_{p}.
\end{split}\end{equation}
According to  \cite[(2.5)]{Zh4},
\begin{align*}
&\int_{\mathbb{R}^{d}}|f(x,y+y')-f(x,y)|^p\d y\leq \|(1-\Delta_y)^{\frac{\beta}{2}}f(x,\cdot)\|^p_p(|y'|^{p\beta}\wedge1),\ \ f\in \H_y^{\bb,p}.
\end{align*}
Then
\begin{align*}
\|(f(\cdot+(0,y'))-f(\cdot))\|^p_p&=\int_{\mathbb{R}^{m+d}}|f(x,y+y')-f(x,y)|^p\d x\d y\\
&\leq \|f\|_{\mathbb{H}_y^{\beta,p}}^p(|y'|^{p\beta}\wedge 1).
\end{align*}
Combining this with Lemma \ref{SET}, for any $\gamma\in(0,\alpha-\frac{m+2d}{p})$, we obtain
\begin{align*}
&\left\|\left\langle (\nabla_\sigma u)(\cdot+(0,y'))-(\nabla_\sigma u)(\cdot),\b(\cdot+(0,y'))-\b(\cdot)\right\rangle\right\|^p_p\\
=&\int_{\mathbb{R}^{m+d}}\left|\left\langle (\nabla_\sigma u)(x,y+y')-(\nabla_\sigma u)(x,y),\b(x,y+y')-\b(x,y)\right\rangle\right|^p\d x \d y\\
\leq &\int_{\mathbb{R}^{m}}\left(\sup_{y\in\mathbb{R}^d}\left|(\nabla_\sigma u)(x,y+y')-(\nabla_\sigma u)(x,y)\right|^p
\int_{\mathbb{R}^{d}}\left|\b(x,y+y')-\b(x,y)\right|^p \d y\right)\d x\\
\preceq& \left\|\nabla_\sigma u\right\|_{\mathbb{H}_\sigma^{\alpha,p}}^p(|y'|^{\gamma p}\wedge 1)\int_{\mathbb{R}^{m+d}} \left|\b(x,y+y')-\b(x,y)\right|^p \d x\d y\\
\preceq& \left\|\nabla_\sigma u\right\|_{\mathbb{H}_\sigma^{\alpha,p}}^p(|y'|^{\gamma p}\wedge 1)\|\b\|^p_{\mathbb{H}_y^{\beta,p}}(|y'|^{p\beta}\wedge 1).
\end{align*}
By \eqref{D_y}, Minkovskii inequality and Lemma \ref{SET}, we have
\begin{align*}
&\|(-\Delta_y)^{\frac{\beta}{2}}\langle \nabla_\sigma u,\b\rangle\|_{p}\\
&\preceq \|\langle (-\Delta_y)^{\frac{\beta}{2}}\nabla_\sigma u,\b\rangle\|_{p}+\|\langle \nabla_\sigma u,(-\Delta_y)^{\frac{\beta}{2}}\b\rangle\|_{p}\\
&+\int_{\mathbb{R}^d}\left\|\left\langle (\nabla_\sigma u)(\cdot+(0,y'))-(\nabla_\sigma u)(\cdot),\b(\cdot+(0,y'))-\b(\cdot)\right\rangle\right\|_p|y'|^{-\beta-d}\d y'\\
&\preceq\|\b\|_{p}\|(-\Delta_y)^{\frac{\beta}{2}}\nabla_\sigma u\|_\infty+\|(-\Delta_y)^{\frac{\beta}{2}}\b\|_{p}\|\nabla_\sigma u\|_\infty+\left\|\nabla_\sigma u\right\|_{\mathbb{H}_\sigma^{\alpha,p}}\|\b\|_{\mathbb{H}_y^{\beta,p}}\\
&\preceq\|\b\|_{p}\|(-\Delta_y)^{\frac{\beta}{2}}\nabla_\sigma u\|_{\mathbb{H}_\sigma^{\alpha,p}}+\|\b\|_{\mathbb{H}_y^{\beta,p}}\left\|\nabla_\sigma u\right\|_{\mathbb{H}_\sigma^{\alpha,p}}.
\end{align*}
Substituting this into \eqref{gbsbet}, we finish the proof.
\end{proof}

It is now ready to prove  the  following regularity estimates for solutions of \eqref{PDEf}.
\begin{thm} \label{T-PDE} Let $p,q\ge 1$ satisfy
\beq\label{3.7'} \ff 2 q+ \ff{m+2d}p <1.\end{equation} For any $f\in C_0^\infty([0,T]\times \R^{m+d})$ and $\ll\ge 0$, \eqref{PDEf} has a unique solution $u^\lambda=Q_\lambda f$, where $Q_\ll f$ is in $\eqref{QL}$ for $P_{s,t}=P_{t-s}.$ Moreover:
\beg{enumerate}
\item[$(1)$] There exists a constant $C>0$ such that
\beq\label{gra2-L0}\beg{split}
&\|\nabla_{\sigma}\nabla_{\sigma}u^\lambda\|_{L_p^q}+\|\nabla_y u^\lambda\|_{L_p^q}+\|(-\Delta_y)^{\frac{1}{4}}\nabla_\sigma u^\lambda\|_{L_p^q}\\
&\leq C\|f\|_{L_p^q},\ f\in C_0^\infty([0,T]\times \R^{m+d}).\end{split}
\end{equation}
For any $\alpha\in(0,1)$ with $\alpha<1-\frac{2}{q}$,
 \begin{equation}\label{u2}\begin{split}
&\|(-\Delta_y)^{\frac{1}{4}} u_t^\lambda\|_{\mathbb{H}_\sigma^{\alpha,p}}+\|\nabla _\sigma u_t^{\lambda}\|_{\mathbb{H}_\sigma^{\alpha,p}}\\
&\leq \phi(\lambda)\|f\|_{L^q_p},\ \ t\in[0,T], f\in C_0^\infty([0,T]\times \R^{m+d})
\end{split}\end{equation}
holds for some decreasing function $\phi:(0,\infty)\to(0,\infty)$ with $\lim_{\lambda\to\infty}\phi(\lambda)=0$.
\item[$(2)$] There exists a constant $C>0$ such that
\begin{equation}\begin{split}\label{u1'}
\|\nabla_y\nabla_\sigma u^{\lambda}\|_{L_p^q}\leq C\|f\|_{\H_y^{\frac{1}{2},p,q}},\ \ \ f\in C_0^\infty([0,T]\times \R^{m+d}).
\end{split}\end{equation}
For any $\alpha\in(0,1)$ with $\alpha<1-\frac{2}{q}$,
\begin{equation}\label{u2'}\begin{split}
&\|\nabla_y u_t^{\lambda}\|_{\mathbb{H}_\sigma^{\alpha,p}}+\|(-\Delta_y)^{\frac{1}{4}} \nabla_\sigma u^{\lambda}_t\|_{\mathbb{H}_\sigma^{\alpha,p}}\\
&\leq \phi(\lambda)\|f\|_{\H_y^{\frac{1}{2},p,q}},\ \ t\in[0,T], f\in C_0^\infty([0,T]\times \R^{m+d})
\end{split}\end{equation}
holds for some  decreasing function $\phi:(0,\infty)\to(0,\infty)$ with $\lim_{\lambda\to\infty}\phi(\lambda)=0$.
\end{enumerate}
\end{thm}
\begin{proof} (a)
By $(A_2)$ and Lemma \ref{L2.1},   \eqref{PDEf} has a unique solution $u^\lambda=Q_\lambda f$ such that
$$\|u^\lambda\|_\mathbb{B}:= \|u^\lambda\|_\infty + \|\nn_\si u^\lambda\|_\infty\le \psi(\ll) \|f\|_{L^q_p}, \ \ \ll\ge 0,$$
holds for some decreasing function $\psi:(0,\infty)\to(0,\infty)$ with $\lim_{\lambda\to\infty}\psi(\lambda)=0$. Let $g\in C_0^\infty(\R^{m+d})$. By $(A_2)$, the heat equation $\pp_t P_t g= \L P_t g$, and the contraction  of $P_{t}$ in $L^p(\mathbb{R}^{m+d})$, we have
\begin{align}\label{Lu00}
\|\L u^\lambda\|_{L^q_p}\preceq\|f\|_{L^q_p}.
\end{align}
Since   $(\nabla_{\sigma}g)_{i}=U_i g$, $(\nabla_{\sigma}\nabla_{\sigma}g)_{ij}=U_i U_j g$, $1\leq i,j\leq m$,    Lemma \ref{L3.1}(2) gives
\begin{align}\label{gra2-L}
\|\nabla_{\sigma}\nabla_{\sigma}g\|_p \preceq\|(-\L)^{\frac{1}{2}}\nabla_{\sigma}g\|_p\preceq\|(-\L)g\|_p.
\end{align}
Combining this with  \eqref{Lu00}, we obtain
\begin{align}\label{Lu0}
\|U_i U_j u^\lambda\|_{L^q_p}\preceq\|f\|_{L^q_p}, \ \ 1\leq i,j\leq m.
\end{align}
Since \eqref{3.1} implies  $U_i U_j=\sum_{l=1}^d (G_l)_{ij}\partial_{y_l}, i\neq j$, it follows from   \eqref{3.2} that
\begin{align}\label{p-y}
\sum_{l=1}^d |\partial_{y_l}g|^2\preceq\sum_{i,j=1}^{m}|U_i U_jg|^2
\end{align}
This together with \eqref{Lu0} leads to
\begin{align}\label{gra-y0}
\|\nabla_y u^\lambda\|_{L^q_p}\preceq\|f\|_{L^q_p}.
\end{align}
On the other hand, \eqref{half} implies
\begin{align}\label{gra-y00}
\|(-\Delta_y)^{\frac{1}{2}} u^\lambda\|_{L^q_p}\preceq\|f\|_{L^q_p}.
\end{align}
Applying  Lemma \ref{L3.1}(5) with $\theta=\frac{1}{2}$ and Young's inequality, we have
\begin{align*}
\|(-\Delta_y)^{\frac{1}{4}} \nabla_\sigma u^{\lambda}\|_{L^q_p}\preceq\|(-\Delta_y)^{\frac{1}{2}}u^\lambda\|_{L^q_p}+\|(-\L)u^\lambda\|_{L^q_p}\preceq  \|f\|_{L^q_p}.
\end{align*}
Combining this with \eqref{Lu0} and \eqref{gra-y0}, we prove \eqref{gra2-L0}.

Next, recall that
$$u^\ll_s=(Q_\ll f)_s:=\int_s^T \e^{-\ll (t-s)}P_{t-s} f_t\d t.$$ By \eqref{GE1}, H\"{o}lder's inequality,   and noticing that $\alpha<1-\frac{2}{q}$ implies $\frac{q}{q-1}(-\frac{\alpha}{2}-\frac{1}{2})>-1$, we obtain
\begin{equation*}\begin{split}
&\|\nabla_\sigma u^\lambda_s\|_{\mathbb{H}_\sigma^{\alpha,p}}+\|(-\Delta_y)^{\frac{1}{4}}u^\lambda_s\|_{\mathbb{H}_\sigma^{\alpha,p}}\\
&\preceq \int_{s}^{T}\e^{-\lambda (t-s)}(t-s)^{-\frac{\alpha}{2}-\frac{1}{2}}\|f_t\|_{p}\d t\\
&\preceq \left(\int_{0}^{T}\e^{-\lambda \frac{q}{q-1}(t-s)}(t-s)^{\frac{q}{q-1}(-\frac{\alpha}{2}-\frac{1}{2})}\d t\right)^{\frac{q-1}{q}}\|f\|_{L^q_p}\\
&=: \phi(\lambda)\|f\|_{L^q_p},
\end{split}\end{equation*}
where $\phi$ is decreasing  with $\lim_{\lambda\to\infty}\phi(\lambda)=0$.
 Therefore, assertion (1) is proved.

(b)
Let $w^\ll=(-\Delta_y)^{\frac{1}{4}}u^\ll$, where $u^\ll:=Q_\ll f$ is the unique solution of  \eqref{PDEf}. We have
\begin{equation}\begin{split}\label{we}
\pp_t w_t^\ll= (\ll-\L) w_t^\ll -(-\Delta_y)^{\frac{1}{4}}f_t,\ \ w_T^\ll=0.
\end{split}\end{equation}
Applying \eqref{half} and  \eqref{gra2-L0} for $(-\DD_y)^{\ff 1 4} f$ replacing $f$, we obtain
\begin{equation*}\begin{split}
\|\nabla_y \nabla_{\sigma} u^\lambda\|_{L^q_p}\asymp \|(-\DD_y)^{\ff 1 2} \nn_\si u^\ll\|_{L_p^q}
=&\|(-\Delta_y)^{\frac{1}{4}}\nabla_\sigma w^\lambda\|_{L^q_p}
\preceq \|(-\Delta_y)^{\frac{1}{4}} f\|_{L^q_p}\preceq\|f\|_{\H_y^{\frac{1}{2},p,q}}.
\end{split}\end{equation*} So, \eqref{u1'} holds.

Finally, applying \eqref{u1'} to $(w^\ll, (-\DD_y)^{\ff 1 4}f)$ replacing $(u^\ll, f)$, we prove
 \eqref{u2'}.
\end{proof}

We now  investigate the regularity of the solution to the following singular equation for $\R^{m+d}$-valued $u_t=(u_t^1,\cdots, u_t^{m+d})$:
\begin{align}\label{PDEuf}
\pp_t u_t= (\ll-\L) u_t-\nabla_{\sigma \mathbf{b}_t} u_t-\sigma \mathbf{b}_t,\ \ u_T=0.
\end{align}

\begin{thm} \label{T-PDE2} Let $p,q\ge 1$ satisfy $\eqref{3.7'}$.
\begin{enumerate}
\item[$(1)$] Assume $\mathbf{b}\in C_0^\infty([0,T]\times \mathbb{R}^{m+d};\mathbb{R}^{m})$.
Then there exists a constant $\lambda_0>0$ such that for any $\lambda\geq\lambda_0$, the equation \eqref{PDEuf} has a unique solution
 $($denoted by $\Xi_\lambda \mathbf{b})$ satisfying
\begin{equation}\begin{split}\label{u1''}
\|\nabla_\sigma\nabla_\sigma \Xi_\lambda \mathbf{b}\|_{L_p^q}\preceq\|\mathbf{b}\|^2_{L_p^q}+\|\mathbf{b}\|_{L_p^q}\|\sigma \mathbf{b}\|_{L_p^q}+\|\sigma \mathbf{b}\|_{L_p^q}.
\end{split}\end{equation}
\item[$(2)$]
There exists a constant $\lambda_1\geq\lambda_0$ such that for any $\lambda\geq\lambda_1$, 
 \begin{equation}\begin{split}\label{u1'''}
\|\nabla_y\nabla_\sigma \Xi_\lambda \mathbf{b}\|_{L_p^q}\preceq\|\sigma \mathbf{b}\|_{\H_y^{\frac{1}{2},p,q}}+\|\sigma \mathbf{b}\|_{\H_y^{\frac{1}{2},p,q}}\|\mathbf{b}\|_{\H_y^{\frac{1}{2},p,q}}.
\end{split}\end{equation}
Moreover, for any $\alpha\in(\frac{m+2d}{p},1-\frac{2}{q})$,
\begin{equation}\begin{split}\label{u2'''-0}
\sup_{t\in [0,T]}\|\nabla_y (\Xi_\lambda \mathbf{b})_t\|_{\H_\sigma^{\alpha,p}}\preceq \phi(\lambda)\left(\|\sigma \mathbf{b}\|_{\H_y^{\frac{1}{2},p,q}}+\|\sigma \mathbf{b}\|_{\H_y^{\frac{1}{2},p,q}}\|\mathbf{b}\|_{\H_y^{\frac{1}{2},p,q}}\right)
\end{split}\end{equation}
holds for some decreasing function $\phi:(0,\infty)\to(0,\infty)$ with $\lim_{\lambda\to\infty}\phi(\lambda)=0$.
\end{enumerate}
\end{thm}
\begin{proof}
(1) By Lemma \ref{L2.1} and Theorem \ref{T2.0}, there exists a constant $\lambda_0>0$ such that for any $\lambda\geq\lambda_0$, the equation \eqref{PDEuf} has a unique solution $u^\lambda (=:\Xi_\lambda \mathbf{b})$.
By \eqref{gra2-L0}, we have
\begin{equation*}\begin{split}
\|\nabla_\sigma\nabla_\sigma \Xi_\lambda \mathbf{b}\|_{L_p^q}&
\preceq\|\nabla_{\sigma \mathbf{b}} \Xi_\lambda \mathbf{b}+\sigma \mathbf{b}\|_{L_p^q}\\
&\preceq \|\mathbf{b}\|_{L_p^q}\|\nabla_\sigma \Xi_\lambda \mathbf{b}\|_\infty+\|\sigma \mathbf{b}\|_{L_p^q}\\
&\preceq\|\mathbf{b}\|^2_{L_p^q}+\|\mathbf{b}\|_{L_p^q}\|\sigma \mathbf{b}\|_{L_p^q}+\|\sigma \mathbf{b}\|_{L_p^q},
\end{split}\end{equation*}

(2) Let $H$ be the space of measurable functions $u:[0,T]\times \mathbb{R}^{m+d}\to\mathbb{R}^{m+d}$ such that
$$\|u\|_H:= \sup_{t\in[0,T]}\left(\left\|\nabla_\sigma u_t\right\|_{\H_\sigma^{\alpha,p}}+\|(-\Delta_y)^{\frac{1}{4}}\nabla_\sigma u_t\|_{\H_\sigma^{\alpha,p}}\right)<\infty.$$
Then $H$ is a Banach space with the norm $\|\cdot\|_{H}$ defined above.
For any $u\in H$, let $\Phi u$ be the solution to the following equation:
\begin{align}\label{PDEuf'}
\pp_t (\Phi u)_t= (\ll-\L) (\Phi u)_t-\nabla_{\sigma \mathbf{b}_t} u_t-\sigma \mathbf{b}_t,\ \ u_T=0.
\end{align} By \eqref{IT2}, we have $\Phi u=Q_\ll(\nn_{\si\b} u-\si\b).$ By Lemma \ref{b-gra} with $\beta=\frac{1}{2}$, we have
\begin{equation}\begin{split}\label{b-gra1}
\|\nabla_{\sigma \mathbf{b}}u\|_{\H_y^{\frac{1}{2},p,q}}
&\preceq \|\mathbf{b}\|_{\H_y^{\frac{1}{2},p,q}}\left(\sup_{t\in[0,T]}\|(-\Delta_y)^{\frac{1}{4}}\nabla_\sigma u_t\|_{\H_\sigma^{\alpha,p}}+\sup_{t\in[0,T]}\left\|\nabla_\sigma u_t\right\|_{\H_\sigma^{\alpha,p}}\right).
\end{split}\end{equation}
Combining this with \eqref{u2} and \eqref{u2'}, we obtain
\begin{equation}\begin{split}\label{one-four}
\|\Phi u\|_H&\preceq \phi(\lambda)\|\nabla_{\sigma \mathbf{b}}u+\sigma \mathbf{b}\|_{\H_y^{\frac{1}{2},p,q}}\preceq \phi(\lambda)\left(\|\sigma \mathbf{b}\|_{\H_y^{\frac{1}{2},p,q}}+\|\mathbf{b}\|_{\H_y^{\frac{1}{2},p,q}}\|u\|_H\right)<\infty.
\end{split}\end{equation}
So,   $\Phi u\in H$ for $u\in H$. Moreover, for any $u$, $\tilde{u}\in H$,   \eqref{PDEuf'} and \eqref{one-four} imply
\begin{equation*}
\|\Phi u-\Phi \tilde{u}\|_{H} \preceq \phi(\lambda)\|\nabla_{\sigma \mathbf{b} }(u-\tilde{u})\|_{\H_y^{\frac{1}{2},p,q}}
 \preceq \phi(\lambda)\|\mathbf{b}\|_{\H_y^{\frac{1}{2},p,q}}\|u-\tilde{u}\|_H.
 \end{equation*}
Since $\phi(\ll)\to 0$ as $\ll\to\infty$, there exists a constant $\lambda_1>0$ such that $\phi(\lambda)\|\mathbf{b}\|_{\H_y^{\frac{1}{2},p,q}}<\frac{1}{2}$ for $\lambda\ge \lambda_1$. Then by the fixed point theorem, for any $\lambda\ge \lambda_1$, the equation  \eqref{PDEuf} has a unique solution $\Xi_\lambda \mathbf{b}\in H.$   Furthermore, \eqref{one-four} implies
\begin{equation}\begin{split}\label{one-four1}
\|\Xi_\lambda \mathbf{b}\|_{H}
&\preceq \phi(\lambda)\|\sigma \mathbf{b}\|_{\H_y^{\frac{1}{2},p,q}},\ \ \lambda\geq \lambda_1.
\end{split}\end{equation}
This together with  \eqref{b-gra1} gives
\begin{equation}\begin{split}\label{b-gra2}
\|\nabla_{\sigma \mathbf{b}}\Xi_\lambda \mathbf{b}\|_{\H_y^{\frac{1}{2},p,q}}\preceq\|\sigma \mathbf{b}\|_{\H_y^{\frac{1}{2},p,q}}\|\mathbf{b}\|_{\H_y^{\frac{1}{2},p,q}}, \ \ \lambda\geq \lambda_1.
\end{split}\end{equation}
Then \eqref{u1'} and \eqref{b-gra2} imply
\begin{equation*}\begin{split}
\|\nabla_y\nabla_\sigma \Xi_\lambda \mathbf{b}\|_{L^q_p}\preceq\|\nabla_{\sigma \mathbf{b}} u^\lambda+\sigma \mathbf{b}\|_{\H_y^{\frac{1}{2},p,q}}\preceq \|\sigma \mathbf{b}\|_{\H_y^{\frac{1}{2},p,q}}+\|\sigma \mathbf{b}\|_{\H_y^{\frac{1}{2},p,q}}\|\mathbf{b}\|_{\H_y^{\frac{1}{2},p,q}}.
\end{split}\end{equation*}
Similarly,   \eqref{u2'} and \eqref{b-gra2} yield
\begin{equation*}\begin{split}
\sup_{t\in [0,T]}\|\nabla_y (\Xi_\lambda \mathbf{b})_t\|_{\H_\sigma^{\alpha,p}}&\preceq \phi(\lambda)\|\nabla_{\sigma \mathbf{b}}u^\lambda+\sigma \mathbf{b}\|_{\H_y^{\frac{1}{2},p,q}}\\
&\preceq \phi(\lambda)\left(\|\sigma \mathbf{b}\|_{\H_y^{\frac{1}{2},p,q}}+\|\sigma \mathbf{b}\|_{\H_y^{\frac{1}{2},p,q}}\|\mathbf{b}\|_{\H_y^{\frac{1}{2},p,q}}\right).
\end{split}\end{equation*}
Then the proof is finished.
\end{proof}

We are now ready to   prove \eqref{D1} and \eqref{D1'}.

\beg{proof}[Proof of \eqref{D1} and \eqref{D1'}]   We first consider smooth $\b$ then extend to the situation of Theorem \ref{T3.1}.

(a) Let  $\mathbf{b}\in C^\infty([0,T]\times \mathbb{R}^{m+d}, \mathbb{R}^{m})$. Then for any $h\in C_0^\infty(\mathbb{R}^{m+d})$, $h\mathbf{b}\in C^\infty([0,T]\times \mathbb{R}^{m+d}, \mathbb{R}^{m})$. Applying Theorem \ref{T-PDE2}, we obtain that 
\begin{equation}\begin{split}\label{u1h}
\|\nabla_\sigma\nabla_\sigma \Xi_\lambda h\mathbf{b}\|_{L_{p}^{q}}\preceq \|h\mathbf{b}\|^2_{L_{p}^{q}}+\|h\mathbf{b}\|_{L_{p}^{q}}\|\sigma h\mathbf{b}\|_{L_{p}^{q}}+\|\sigma h\mathbf{b}\|_{L_{p}^{q}}
\end{split}\end{equation}
and
\begin{equation}\begin{split}\label{u1'''h}
\|\nabla_y\nabla_\sigma \Xi_\lambda h\mathbf{b}\|_{L_{p}^{q}}\preceq\|\sigma h\mathbf{b}\|_{\H_y^{\frac{1}{2},p,q}}+\|\sigma h\mathbf{b}\|_{\H_y^{\frac{1}{2},p,q}}\|h\mathbf{b}\|_{\H_y^{\frac{1}{2},p,q}}
\end{split}\end{equation}
for any $\lambda\geq \lambda_1$.
Next,  by Lemma \ref{L3.1}(4),  \eqref{u2'''-0},  and Theorem \ref{T2.0}(1),
\begin{equation}\begin{split}\label{u2'''-0h}
\|\nabla_y (\Xi_\lambda h\mathbf{b})\|_{\infty}\preceq \phi(\lambda)\left(\|\sigma h\mathbf{b}\|_{\H_y^{\frac{1}{2},p,q}}+\|\sigma h\mathbf{b}\|_{\H_y^{\frac{1}{2},p,q}}\|h\mathbf{b}\|_{\H_y^{\frac{1}{2},p,q}}\right)
\end{split}\end{equation}
and
\begin{equation}\begin{split}\label{sigma-0h}
\|\nabla_\sigma (\Xi_\lambda h\mathbf{b})\|_{\infty}\leq \phi(\lambda)(\|\sigma h\mathbf{b}\|_{L^{q}_{p}}+\|h\mathbf{b}\|_{L^{q}_{p}})
\end{split}\end{equation}
hold  for large $\ll>0$ and  some decreasing function $\phi:(0,\infty)\to(0,\infty)$ with $\lim_{\lambda\to\infty}\phi(\lambda)=0$.

Moreover,    for any $R>0$, there exists a constant $c(R)>0$ such that
\begin{equation}\begin{split}\label{K-N0}
|\nabla f|^2(x)&\leq c(R)\left(\sum_{i=1}^m|U_i f|^2(x)+\sum_{i=1}^d|\partial_{y_i} f|^2(x)\right)\\
&=c(R)\left(|\nabla_\sigma f|^2(x)+|\nabla_yf|^2(x)\right),\ \ |x|\leq R, f\in C^1(\R^{m+d}).
\end{split}\end{equation}
Combining \eqref{u2'''-0h}-\eqref{K-N0}, we conclude that for large $\ll>0$,
\begin{equation}\begin{split}\label{nabla h}
\|h\nabla (\Xi_\lambda h\mathbf{b})\|_{\infty}\leq C_{\sigma,h}\|h\|_\infty\phi(\lambda)(\|h\mathbf{b}\|_{\H_y^{\frac{1}{2},p,q}}+\|h\mathbf{b}\|^2_{\H_y^{\frac{1}{2},p,q}}),
\end{split}\end{equation}
where $C_{\sigma,h}>0$ is a constant depending on $\mathrm{supp} h$ and $\|\sigma 1_{\mathrm{supp} h}\|_{\infty}$. Similarly, \eqref{u1h},  \eqref{u1'''h} and \eqref{K-N0} imply
\begin{equation}\begin{split}\label{nabla^2 h}
\|\nabla h\nabla_\sigma (\Xi_\lambda h\mathbf{b})\|_{L^{q}_{p}}\leq C_{\sigma,h}\Big(\|h\|_\infty+\|h'\|_\infty)(\|h\mathbf{b}\|_{\H_y^{\frac{1}{2},p,q}}+\|h\mathbf{b}\|^2_{\H_y^{\frac{1}{2},p,q}}\Big)
\end{split}\end{equation} for large $\ll>0$ and some
 constant $C_{\sigma,h}>0$ depending on $\mathrm{supp} h$ and $\|\sigma 1_{\mathrm{supp} h}\|_{\infty}$.
 Therefore, \eqref{D1} and \eqref{D1'} are proved.

(b) Now,  assume that for any $h\in C_0^\infty(\R^{m+d})$ we have
\beq\label{*H2}\|h \mathbf{b}\|_{\H_y^{\frac{1}{2},p,q}}=\|(1-\Delta_y)^{\ff 1 4} (h\b)\|_{L^{q}_{p}}<\infty.\end{equation}
Let $\rho$ be a non-negative smooth function with compact support in $\mathbb{R}^{m+d}$ and
$\int_{\mathbb{R}^{m+d}}\rho(z)\d z=1$.
For any $n\in \mathbb{N}$, let
\begin{equation}\begin{split}\label{Ap}
\rho_n(z)=n^{m+d}\rho(nz),\ \ \mathbf{b}^n=\rho_n\ast (h\mathbf{b}), \ \ z\in\mathbb{R}^{m+d}.
\end{split}\end{equation}
Then
$$\lim_{n\to\infty}\|\mathbf{b}^n-h\mathbf{b}\|_{\H_y^{\frac{1}{2},p,q}}=0.$$ Combining this with \eqref{nabla h} and \eqref{nabla^2 h} for $\b^n$ replacing $h\b$, and by an approximation method, we may find out a   constant $\lambda_1>0$  not depending on $n$,  such that for any $\lambda \geq \lambda_1$, the unique solution $u^\lambda (=:\Xi_\lambda h\mathbf{b})$ of \eqref{PDEuf}  satisfies
\begin{equation}\begin{split}\label{nabla hge}
\|h\nabla (\Xi_\lambda h\mathbf{b})\|_{\infty}\preceq (1+C_{\sigma,h})\|h\|_\infty\phi(\lambda)\big(\|h\mathbf{b}\|_{\H_y^{\frac{1}{2},p,q}}+\|h\mathbf{b}\|^2_{\H_y^{\frac{1}{2},p,q}}\big),
\end{split}\end{equation}
and
\begin{equation}\begin{split}\label{nabla^2 hge}
\|\nabla h\nabla_\sigma (\Xi_\lambda h\mathbf{b})\|_{L^{q}_{p}}\preceq (1+C_{\sigma,h})(\|h\|_\infty+\|h'\|_\infty)\big(\|h\mathbf{b}\|_{\H_y^{\frac{1}{2},p,q}}+\|h\mathbf{b}\|^2_{\H_y^{\frac{1}{2},p,q}}\big),
\end{split}\end{equation}
here, $C_{\sigma,h}$, $\phi$ are in \eqref{nabla h} and \eqref{nabla^2 h}. Combining these with \eqref{*H2}, we finish the proof. \end{proof}

\paragraph{Acknowledgement.} The authors would like to thank the referee  for corrections and helpful comments.

\end{document}